%% file: matrixIIIarxiv1.tex
\documentclass[reqno]{amsart}
  \input amssymb.sty

\input tropMacro.tex
\usepackage{pstricks,pst-node}
\def\hfA{\Inu{f_A}}

\def\tlv{\tilde v}
\def\tlV{\widetilde V}

\def\tlB{\widetilde B}
\def\hV{\widehat V}
\def\hB{\widehat B}
\def\tt{\eta}
\def\tu{u}
\def\bbt{\overline \bt}

\newcommand{\Inu}[1]{\widehat{#1}}

\def\mzero{(\zero)}
\def\vzero{\overline{\zero}}

\def\tlmu{\tilde \mu}
\def\GP{A}
\def\core{\operatorname{core}}
\def\mV{\mathcal V}
\def\tancore{\operatorname{tcore}}
\def\tcore{\operatorname{tcore}}
\def\tancoreA{\tancore(A)}
\def\Vacore{\tV(\acore(A))}
\def\acore{\operatorname{anti-tcore}}
\def\coreA{\core(A)}
\def\tcoreA{\tancore(A)}
\def\acoreA{\acore(A)}

\def\semiring0{semiring$^\dagger$}
\def\semirings0{semirings$^\dagger$}
\def\domain0{domain$^\dagger$}
\newcommand{\weight}[1]{w(#1)}

\newcommand{\len}[1]{\ell(#1)}

\def\domains0{domains$^\dagger$}
\def\field0{semifield$^\dagger$}
\def\fields0{semifields$^\dagger$}

\def\Gker{\operatorname{g-ker}}

\def\hnu{\hat \nu}
\def\Gker{\operatorname{g-ker}}

\def\pSkip{\vskip 1.5mm \noindent}
\def\essn{{\operatorname{es}}}

\def\pipeGS{{\underset{\operatorname{\, gs }}{\mid}}}
\def\lmod{\mathrel   \pipeGS \joinrel\joinrel \joinrel =}
\def\lmodg{\lmod}

\def\ggsim{\, \, \curlyvee \,}
\def\gsim{ {\underset{\operatorname{gd}}{\ggsim }} }

\def\nucong{\cong_\nu}
\def\nnucong{\not \cong_\nu}
\def\nuge{\ge_\nu}

\newcommand{\etype}[1]{\renewcommand{\labelenumi}{(#1{enumi})}}
\def\eroman{\etype{\roman}}
\def\ealph{\etype{\alph}}

\newcommand{\Det}[1]{ |{#1}|}

\def\tGz{\mathcal G_\zero}
\def\tTz{\mathcal T_\zero}

\def\({\left(}
\def\){\right)}

\def\multi{multicycle}

\def\scycle{scycle}

\def\trn{{\operatorname{t}}}

\def\cyc{C}

\def\grph{G}
\def\igrph{\widetilde{G}}

\def\quasi{quasi}

\def\tGz{\mathcal G_\zero}

\def\ghost{\text{ghost}}

\def\a{\alpha}
\def\la{\lambda}
\def\sig{\sigma}

\def\one{\mathbb{1}}
\def\zero{\mathbb{0}}

\def\regular{nonsingular}

\def\um{I}

\def\id{\operatorname{id}}
\def\ef{f^{\operatorname{es}}}

\newcommand{\trace}[1]{\operatorname{tr}(#1)}

\def\rone{\one_R}
\def\rzero{\zero_R}

\def\fzero{\zero_F}

\def\la{\lambda}

\newtheorem{thm}[theorem]{Theorem}
\newtheorem*{thm*}{Theorem}
\newtheorem*{dig*}{Digression}

\newtheorem{cor}[theorem]{Corollary}

\newtheorem{lem}[theorem]{Lemma}
\newtheorem{rem}[theorem]{Remark}
\newtheorem{prop*}{Proposition}

\newtheorem{prop}[theorem]{Proposition}
\newtheorem{defn}[theorem]{Definition}
\newtheorem*{examp*}{Example}
\newtheorem*{examples*}{Examples}
\newtheorem*{remark*}{Remark}

\newtheorem*{defn*}{Definition}
\newtheorem*{note*}{Note}

\begin{document}
\title[Supertropical Matrix Algebra III: Powers of matrices and generalized eigenspaces]
{Supertropical Matrix Algebra III: \\Powers of matrices and
generalized eigenspaces}

\author{Zur Izhakian}
\address{Department of Mathematics, Bar-Ilan University, Ramat-Gan 52900,
Israel} \email{zzur@math.biu.ac.il}
\author{Louis Rowen}
\address{Department of Mathematics, Bar-Ilan University, Ramat-Gan 52900,
Israel} \email{rowen@macs.biu.ac.il}

\thanks{This research is supported  by the
Israel Science Foundation (grant No.  448/09).}

\thanks{The first author has been a Leibniz Fellow in the
Oberwolfach Leibniz Fellows Programme (OWLF), Mathematisches
Forschungsinstitut Oberwolfach, Germany.}



\date{\today}



\begin{abstract}
We investigate powers of supertropical matrices, with special
attention to the role of the coefficients of the  supertropical
 characteristic polynomial (especially the supertropical trace) in controlling the rank of a power of a
 matrix. This  leads to a Jordan-type decomposition
  of  supertropical matrices, together with
  a generalized eigenspace decomposition of a power of an
 arbitrary supertropical matrix.
\end{abstract}

\maketitle




\section{Introduction}
\numberwithin{equation}{section}

This paper develops further the tropical matrix theory from the
point of view of supertropical algebras, whose foundation  was
laid out in \cite{IzhakianRowen2008Matrices} and further analyzed
in \cite{IzhakianRowen2009Equations};  a treatment from the point
of view of linear algebra is to be given in
\cite{IzhakianKnebuschRowen2009Linear}.

We recall briefly the main idea behind  supertropical matrix
theory. Although matrix theory over the max-plus semiring is
hampered by the lack of negatives, the use of the ``ghost ideal''
in supertropical domains enables one to recover most of the
classical spirit (and theorems) of matrix theory, and some of the
proofs actually become easier than in the classical case. For
example, instead of taking the determinant, which requires $-1$,
one defines the \textbf{supertropical determinant} to be the
permanent, and we define a matrix to be \textbf{nonsingular} when
its supertropical determinant is not a ghost.  Likewise, vectors
are called  \textbf{tropically dependent} when some linear
combination (with tangible coefficients) is a ghost vector;
otherwise, they are \textbf{tropically independent}.

Prior supertropical results include the fact that the rows (or
columns) of a matrix are tropically dependent iff the matrix is
nonsingular (\cite[Theorem~6.5]{IzhakianRowen2008Matrices}), and,
more generally, the maximal number of tropically independent rows
(or columns) is the same as the maximal size of a nonsingular
submatrix, cf.~\cite{IzhakianRowen2009TropicalRank}. A key tool is
the use of \textbf{\quasi-identity} matrices, defined below as
nonsingular multiplicative  idempotent matrices equal to the
identity matrix plus a ghost matrix. Quasi-identity matrices were
obtained in \cite[Theorem~2.8]{IzhakianRowen2009Equations} by
means of the adjoint, which was used in \cite[Theorems~3.5 and
~3.8]{IzhakianRowen2009Equations} to solve equations via a variant
of Cramer's rule, thereby enabling us to compute supertropical
eigenvectors in \cite[Theorem~5.6]{IzhakianRowen2009Equations}.
 Furthermore, any matrix $A$ satisfies its \textbf{characteristic
polynomial}~ $f_A:= |\la I + A|$, in the sense that $f_A(A)$ is
ghost (\cite[Theorem~5.2]{IzhakianRowen2008Matrices}), and the
tangible roots of $f_A$ turn out to be the supertropical
eigenvalues of $A$
(\cite[Theorem~7.10]{IzhakianRowen2008Matrices}).
However, something seems to go wrong, as indicated in
 \cite[Example~5.7]{IzhakianRowen2009Equations}, in which it is
 seen that even when the characteristic polynomial is a product of
 distinct tangible linear factors, the supertropical eigenvectors need not
 be tropically independent.

This difficulty can be resolved by passing to asymptotics, i.e.,
high enough powers of $A$. In contrast to the classical case, a
power of a nonsingular $n\times n$ matrix can be singular (and
even ghost). Asymptotics of matrix powers have been studied
extensively over the max-plus algebra, as described in
\cite[Chapter 25.4]{ABG}, but the situation is not quite the same
in the supertropical context, since ``ghost entries'' also play a
key role. Whereas \cite{IzhakianRowen2008Matrices} and
\cite{IzhakianRowen2009Equations} focused on the supertropical
determinant, nonsingular matrices and the adjoint, it turns out
that the simple cycles contributing to the first coefficient $\a
_{\mu}$ of the supertropical characteristic polynomial, together
with other simple cycles of the same average weight, provide the
explanation for the phenomena discussed in this paper. This can be
described in terms of the \textbf{supertropical trace}. Thus, this
paper involves a close study of the cycles of the graph of the
matrix, one of the main themes of \cite[Chapter~25]{ABG}.
Nevertheless, the supertropical point of view leads to somewhat
more general results, which were not previously accessible in the
language of the max-plus algebra.

 There is a reduction to the case where the graph of an $n \times n$ matrix $A$ is strongly connected, in which case
the following results hold, cf.~Theorems~\ref{singp} and
\ref{ghostpot1}:
\begin{enumerate} \ealph
    \item
When all of these leading simple cycles are tangible and disjoint,
and their vertex set contains all $n$ vertices, then every power
of the matrix $A$ is nonsingular. \pSkip

\item When (a) does not hold and there exists a leading tangible simple cycle,
disjoint from all the other leading cycles, then every power of
$A$ is non-ghost, but some power is singular. \pSkip

\item When the weight of each of the disjoint leading simple cycles is
ghost, then some power of $A$ is ghost.
\end{enumerate}

 As indicated above, the ultimate objective of this paper is to show that
 the pathological
 behavior described in
 \cite[Example~5.7]{IzhakianRowen2009Equations} can be avoided by
 passing to high enough powers of the matrix $A$. In general, there is some
 change in the behavior of powers up to a certain power $A^m$, depending on the matrix $A$, until
  the theory ``stabilizes;'' we call this $m$ the \textbf{stability index}, which  is the analog of the ``cyclicity'' in  \cite[Chapter 25]{ABG}.
 We see this behavior in the characteristic
 polynomial, for example, in Theorem~\ref{char1} and Lemma~\ref{char3}.
 The  stability index is understood in terms of the leading
 simple cycles of the graph of $A$, as explained in
 Theorem~\ref{corepower}. A key concept here is the ``tangible core'' of the digraph of a
 matrix~$A$, which is the aggregate of those  simple cycles which are tangible and disjoint from the others.

   Once stability is achieved, the supertropical theory behaves beautifully. Some
 power of $A$ can be put in full block triangular form, where each
 block $B_i$ satisfies $B_i^2 = \bt_i  B_i$ for some tangible scalar $\bt_i$
 (Corollary~\ref{semiid}) and the
 off-diagonal blocks also behave similarly, as described in Theorem~\ref{stabil}, which
 might be considered our main result.
The special case  for a matrix whose digraph is strongly connected
is a generalization of \cite{CDQV} to supertropical
  matrices. (One can specialize to the max-plus algebra and thus
rederive their theorem.)  These considerations also provide a
Jordan-type decomposition for supertropical
 matrices (Theorem~\ref{Jord}).

Passing to powers of $A$ leads us to study generalized
eigenspaces. A tangible vector $v$
 is a \textbf{generalized supertropical eigenvector} of $A$ if $A^kv$ equals $\bt^k v$ plus a
 ghost, i.e. $A^k v = \bt^k v + \ghost,$
 for some tangible $\bt $ and some $k \in \Net$. (We also include the possibility that  $A^k v $ is ghost.) Again, in contrast to the classical
 theory, the supertropical eigenvalues may change as we pass to higher powers of $A$,
 and the theory only becomes manageable  when we reach a high enough power $A^m$ of $A$.
 In this case, some ``thick'' subspace of $R^{(n)}$ is a direct sum
of generalized eigenspaces of $A^m$, cf.~Theorems~\ref{eigendec}
and ~\ref{eigendec1} (although there are counterexamples for $A$
itself). There is a competing concept of ``weak'' generalized
supertropical eigenvectors,   using ``ghost dependence,'' which we
consider briefly at the end in order to understand the action of
the powers $A^k$ for $k<m,$ as indicated in Theorem~\ref{weak2}.
But this transition is rather subtle, and merits further
investigation.

\section{Supertropical structures}
We recall various notions  from
\cite{IzhakianKnebuschRowen2009Linear,IzhakianRowen2007SuperTropical}.

A \textbf{semiring without zero}, which we notate as \semiring0,
is
 a structure $(R ,+,\cdot, \rone)$ such that $(R ,\cdot \,
,\rone)$ is a multiplicative monoid, with unit element~$\rone$,
and $(R ,+)$ is an additive commutative semigroup, satisfying
distributivity of multiplication over addition on both sides.

 We recall
that the underlying supertropical structure is a
 \textbf{\semiring0 with ghosts}, which  is a
triple $(R,\tG,\nu),$ where $R$ is a \semiring0 and $\tG
 $ is a semigroup ideal, called the
\textbf{ghost ideal}, together with an idempotent map
$$\nu : R \ \to \ \tG $$  called the \textbf{ghost map}, i.e.,
which preserves multiplication as well as addition and
 the key property \begin{equation}\label{supertr}\nu
(a) = a+a. \end{equation} Thus, $\nu(a) = \nu(a) + \nu(a)$ for all
$a  \in R $.

We write $a^{\nu }$ for $\nu(a)$, called the $\nu$-\textbf{value}
of $a$. Two elements $a$ and $b$ in $R$ are said to be
$\nu$-\textbf{equivalent}, written   $a \nucong b$, if $a^\nu =
b^\nu$.  (This was called ``$\nu$-matched'' in
\cite{IzhakianRowen2008Matrices}.) We write $a \nuge b$, and say
that $a$ \textbf{dominates} $b$, if $a^\nu \ge b^\nu$. Likewise we
say that $a$ \textbf{strictly dominates} $b$, written $a >_ \nu
b$, if $a^\nu
>  b^\nu$.

 We
define the relation $\lmodg$, called ``\textbf{ghost surpasses},''
on any semiring with ghosts $R$, by $$b \lmodg a \qquad \text{ iff
} \qquad b=a \quad \text{or} \quad b = a +
  \ghost.$$

We write $b \gsim a$  if $a+b \in \tG$, and say that $a$ and $b$
are \textbf{ghost dependent}.

A \textbf{supertropical \semiring0}  has the extra properties for
all $a,b$:
\begin{enumerate} \eroman
 \item  $a+b   =  a^{\nu } \quad \text{if}\quad  a
\nucong b$; \pSkip
 \item $a+b  \in \{a,b\}\quad \text{if}\quad  a
\nnucong b.$ 
\end{enumerate}

A \textbf{supertropical \domain0} is a supertropical \semiring0
for which   the \textbf{tangible elements} $\tT = R\setminus \tG$
 is a cancellative
monoid and the map $\nu _\tT : \tT \to \tG$ (defined as the
restriction from $\nu$ to $\tT$) is onto. In other words, every
element of $\tG$ has the form $a^\nu$ for some $a\in \tT$.

We also define a \textbf{supertropical \field0} to be a
supertropical \domain0\ $(R, \tG, \nu)$ in which every  tangible
element of $R$ is invertible; in other words, $\tT$ is a group. In
this paper we always assume that $R$ is a supertropical \field0\
which is \textbf{divisible} in the sense that $\root n \of a \in
R$ for each $a \in R.$ With care, one could avoid these
assumptions, but there is no need since a supertropical \domain0\
can be embedded into a divisible supertropical \field0, as
explained in \cite[Proposition~3.21
and~Remark~3.23]{IzhakianRowen2007SuperTropical}.

%

Although in general, the map $\nu: \tT \to \tG$ need not be 1:1,
we define a function $$\hnu : \tG \to \tT$$ such that $\nu \circ
\hnu = \id _{\tG}$, and write $\hat{ b}$ for $\hnu(b)$. Thus,
$(\hat
 b)^\nu =  b$ for all $b\in \tG$. In \cite[Proposition~1.6]{IzhakianRowen2009Equations}, it is shown that $\hnu$ can be taken to be multiplicative on the ghost elements, and we assume this implicitly throughout.

 It often is convenient to obtain a semiring by formally adjoining a zero element $\rzero$ which is considered to be
 less than all other elements of $R$. In this case, $R$ is a semiring with zero element,
$\zero_R$, (often identified in the examples with $-\infty$ as
indicated below), and the \textbf{ghost ideal} $\tGz = \tG \cup \{
\zero_R \} $ is a semiring  ideal.  We write $\tTz$ for $\tT \cup
\{ \rzero\}$. Adjoining  $\rzero$ in this way to a supertropical
\domain0 (resp.~supertropical \field0) gives us a
\textbf{supertropical domain} (resp.~\textbf{supertropical
semifield}.)

We also need the following variant of the \textbf{Frobenius
property}:

\begin{prop}\label{Frob} Suppose $ab \gsim ba$ in a \semiring0 with ghosts.
Then $(a+b)^m \lmodg a^m + b^m$ for all $m$.\end{prop}
\begin{proof} Any term other than $a^m$ or $b^m$ in the expansion of $(a+b)^m$ has the form $a^{i_1}b^{j_1}\cdots$ or
$b^{i_1}a^{j_1}\cdots$ where $i_1 , j_1 \ge 1.$ But then we also
have the respective terms $$a^{i_1-1}bab^{j_1}\cdots \qquad \text{
or } \qquad  b^{i_1-1}aba^{j_1}\cdots \quad,$$ so summing yields
$$a^{i_1-1}(ab+ba) b^{j_1}\cdots \quad \text{ or } \quad
b^{i_1-1}(ab+ba)a^{j_1}\cdots$$ respectively, each of which by
hypothesis are in $\tG$. It follows that the sum of all of these
terms are in ~$\tG$. (We do not worry about duplication, in view
of Equation \eqref{supertr}. )\end{proof}


\begin{prop}\label{Frob1} If $q = dm$ for $d>1,$ and $a,b$ commute in a \semiring0 with ghosts,
then $$(a+b)^q \lmodg \left( a^{m }+ b^{m}\right)^{d-1} \bigg(\sum
_{j = 0}^m a^ {j} b^{m-j}\bigg),$$ with both sides
$\nu$-equivalent. \end{prop}
\begin{proof} Both sides are $\nu$-equivalent to $a^q + b^q + (\sum_j a^j b^{q-j})^\nu,$
and the left side has more ghost terms.\end{proof}

We usually use the algebraic semiring notation (in which $\rzero,
\rone$ denote the respective additive and multiplicative
identities of $R$), but for examples occasionally use
``logarithmic notation,'' in which  $\rone$ is~$0$  and $\rzero$
is $-\infty$. (Our main example is the extended tropical semiring
in which $\tT = \tG = \Real$,  cf.~\cite{zur05TropicalAlgebra}.)

\section{Supertropical matrices}

\subsection{Background on matrices}

For any \semiring0 $R$, we write $M_n(R)$ for the set of $n \times
n$ matrices with entries in ~$R$, endowed with the usual matrix
addition and multiplication. The size of the matrices is always
denoted $n$ throughout this paper. When $R$ is a \semiring0 with
ghosts, we have the  ghost map $$\nu_*: M_n(R) \to M_{n}(\tGz)$$
 obtained by applying $\nu$ to each matrix entry.

 When $\rzero
\in R,$ we write $(\rzero)$ for the zero matrix of $M_n(R)$.
Technically speaking, we need $R$ to have a zero element in order
to define the identity matrix  and obtain a multiplicative unit in
 $M_n(R)$, which is itself a semiring. However, in
\cite{IzhakianRowen2008Matrices} we defined a
\textbf{\quasi-identity} matrix $\um_\tG$ to be a \regular \
multiplicatively idempotent matrix with $\rone$ on the diagonal
and ghosts off the diagonal, and saw that quasi-identity matrices
play a more important role in the supertropical theory than
identity matrices.

The reader should be aware that the formulations of the results
become more complicated when we have to deal with $\rzero$, which
often has to be handled separately. (See, for example, the use of
\cite[Proposition ~6.2]{IzhakianRowen2008Matrices} in proving
\cite[Theorem~6.5]{IzhakianRowen2008Matrices}.) The use of
$\rzero$ leads us to consider strongly connected components in
\S\ref{graph}, and reducible matrices in \S\ref{ssec:bform}.

Let us illustrate the Frobenius property (Proposition \ref{Frob})
for matrices.

\begin{example} Suppose $b =  a^2$ and let
$$ A = \(  \begin{matrix} \zero & a \\
a & \zero
\end{matrix}\), \quad B= \(  \begin{matrix} a & \zero \\
\zero & a
\end{matrix}\). $$
Then
$$ A^2 = B^2 = \(  \begin{matrix} a^2 & \zero \\
\zero & a^2
\end{matrix}\)= \(  \begin{matrix} b & \zero \\
\zero & b
\end{matrix}\), \quad A^2 + B^2 = \(  \begin{matrix} b^\nu & \zero \\
\zero & b^\nu
\end{matrix}\), \quad   (A + B)^2 =  \(  \begin{matrix} a & a \\
a & a
\end{matrix}\) ^2 = \(  \begin{matrix} b^\nu & b^\nu \\
b^\nu & b^\nu
\end{matrix}\).$$

\end{example}

 We define the supertropical determinant $\Det{A}$ to be the
permanent, i.e.,
$$ |A| = \sum_{\sig \in S_n} a_{1, \sig(1)} \cdots a_{n, \sig(n)},$$
as in \cite{zur05TropicalAlgebra},
\cite{IzhakianRowen2009TropicalRank}, and
\cite{IzhakianRowen2008Matrices}. Then
\begin{equation}\label{matmul} \Det{AB} \lmodg
\Det{A}\Det{B},\end{equation} by \cite[Theorem
3.5]{IzhakianRowen2008Matrices}; a quick proof was found by
\cite{AGG}, using their metatheorem which is used to obtain other
ghost-surpassing identities, as quoted in
\cite[Theorem~2.4]{IzhakianRowen2009Equations}.

We say that the matrix $A$ is \textbf{nonsingular} if $| A |$ is
tangible (and thus invertible when $R$ is a supertropical
semifield \cite{IzhakianRowen2008Matrices}); otherwise,  $| A |
\in \tGz$ (i.e., $|A| \lmodg \fzero$) and we say that $A$ is~
\textbf{singular}. Thus, Equation \eqref{matmul} says that
$\Det{AB} = \Det{A}\Det{B}$ when $AB$ is nonsingular, but there
might be a discrepancy when $\Det {AB}\in \tGz.$

One might hope that the ``ghost error'' in Formula \eqref{matmul}
might be bounded, say in terms of $\Det {AB}.$ But we have the
following easy counterexample.
\begin{example}\label{power1} Let $$A = \(  \begin{matrix} a & \rone \\
\rone & \rzero
\end{matrix}\),  \quad a >_\nu \rone. \qquad    \text{Then } \
A^2 =  \left( \begin{matrix} a^2  & a \\
a & \rone\end{matrix}\right),$$  whose determinant is $(a^2)^\nu$
whereas $|A| =\rone$ (the multiplicative unit). Thus we have no
bound for  $\frac{|A^2|}{|A|^2} = (a^2)^\nu$, although $|A^2|
\lmodg |A|^2$.
\end{example}

We also need the following basic fact.

\begin{prop}\label{rmk:quasisingular0} Any multiplicatively
idempotent, nonsingular matrix $A = (a_{i,j})$ over a
supertropical domain is already a quasi-identity matrix.\end{prop}
\begin{proof} $ |A^2| = |A|$ is tangible,
implying $ |A^2| = |A|^2,$ by Equation \eqref{matmul}, and thus $
|A|=\rone$. The $(i,j)$ entry of $ A^2$ is $a _{i,j} = \sum_k
a_{i,k} a_{k,j} = a_{i,i} a_{i,j} + a_{i,j} a_{j,j} + \sum_{k \neq
i,j} a_{i,k} a_{k,j}$. Thus, for  $i=j$ we have $a _{i,i} \le _\nu
a_{i,i}^2$.

On the other hand, by impotency of $A$,  $a_{i,i}a_{i,j} \le_\nu
a_{i,j}$, implying each $a_{i,i}$ is tangible (since otherwise
$a_{i,j}\in \tGz$ for each $j,$ implying $A$ is singular, contrary
to hypothesis).

Also, taking $i=j$ yields $a_{i,i}^2 \le_\nu a_{i,i}$, implying $a
_{i,i} \nucong a_{i,i}^2,$
 and since $a _{i,i} \ne \rzero$ is tangible, we must have each $a
_{i,i} = \rone.$ But then for $i \ne j$ we now have $a _{i,j} =
a_{i,j}^\nu + \sum_{k \neq i,j} a_{i,k} a_{k,j},$ and thus
$a_{i,j}$ is a ghost.
\end{proof}

\begin{rem}\label{rmk:quasisingular}
It easy to verify that changing one (or more) of the diagonal
entries of a quasi-identity matrix $I_\tG$ to be $\rone^\nu$, we
get a singular  idempotent matrix $J_\tG$ with $J_\tG \lmod
I_\tG$.
\end{rem}


\subsection{The weighted digraph}\label{graph}
As described in \cite{ABG}, one major computational tool in
tropical matrix theory is the \textbf{weighted digraph} $\grph _A
= (\tV, \mathcal E)$ of an $n\times n$ matrix $A = (a_{i,j})$,
which is defined to have vertex set $\tV =\{ 1, \dots, n\}$ and an
edge $(i,j)$ from $i$ to $j$  (of \textbf{weight} $a_{i,j}$)
whenever $a_{i,j} \ne \rzero$. We write $\# (\tV)$ for the number
of elements in the vertex set $\tV$.

 As usual, a \textbf{path} $p$ (called ``walk'' in \cite{ABG})
 of  \textbf{length} $\ell = \len{p}$ in a graph
 is a sequence  of
 $\ell$ edges $(i_1, i_2), (i_2, i_3),  \dots,  (i_\ell,
 i_{\ell+1}$), which can
 also be viewed as a sequence of the vertices $(i_1, \dots, i_{\ell
 +1})$.
For example,
 the path $(1,2,4)$ starts at vertex 1, and then proceeds to 2, and finally 4.
 A \textbf{cycle} is a path with the same initial and terminal
 vertex. Thus, $(1,2,4,3,2,5,1)$ is a cycle.

 We say that vertices $i,j \in \mathcal
V(\grph _A)$ are \textbf{connected} if there is
 a path from $i$ to $j$; the vertices $i$ and $j$  are \textbf{strongly
 connected} if there is a
 cycle containing both $i$ and $j$; in other words, there is a path from  $i$ to $j$ and a   path from  $j$ to $i$.
The \textbf{strongly connected component} of a vertex $i$ is the
set of vertices strongly connected to $i$.

 The matrices $A$ that are easiest to deal with are those for
 which the entire graph
$\grph _A$ is strongly connected, i.e., any two vertices are
contained in a cycle. Such matrices are called
\textbf{irreducible}, cf.~\cite{ABG}.

Reducible  matrices   are  an ``exceptional'' case which we could
avoid when taking matrices over supertropical \domains0, i.e.,
domain without zero. Nevertheless, in order to present our results
as completely as we can, we assume from now on that we are taking
matrices over a supertropical domain $R$ (with $\rzero$).

 Compressing each strongly connected component to a vertex, one
obtains the induced \textbf{component digraph} $\igrph_A$ of
$\grph_A$, and thus of $A$,  which is an acyclic digraph. The
number of vertices of $\igrph_A$ equals the number of strongly
connected components of $\grph_A$. Note that the graph $\igrph_A$
is connected iff $\grph_A$ is connected.

 A \textbf{simple cycle}, written as \textbf{\scycle}, is a cycle
having in-degree and out-degree $1$ in each of its vertices
\cite[\S 3.2]{IzhakianRowen2008Matrices}.  For example, the cycle
(1,3,1) is simple, i.e.,  whereas the cycle (1,3,5,3,1) is not
simple.

 A
$k$-\textbf{multicycle} of $\grph _A$ is a disjoint union of
\scycle s the sum of whose lengths is $k$.
 The \textbf{weight} of a
path $p$, written $\weight{p}$,  is the product of the weights of
its edges (where we use the semiring operations); the
\textbf{average weight} of a path $p$ is
$\sqrt[\len{p}]{\weight{p}}$. A path $p$ is called
\textbf{tangible} if $w(p)\in \tT;$ otherwise, $p$ is called
\textbf{ghost}.

Given a subgraph $\grph'$ of $\grph_A$, we write $\tV(\grph')$ for
the set of vertices of $\grph'$. Given a \scycle \ $\cyc$ passing
through vertices $i,j$, we write $\cyc(i,j)$ for the subpath of
$\cyc$ from $i$ to $j$. Thus, the cycle $\cyc$ itself can be
viewed as the subgraph $\cyc(i,i)$ for any vertex $i \in
\tV(\cyc)$.

By \textbf{deleting} a \scycle\ $C = \cyc(i,i)$ from a path $p$,
we mean replacing $\cyc$ by the vertex $i$. For example, deleting
the cycle $(3,6,3)$ from the path $(1,2,3,6,3,5)$ yields the path
$(1,2,3,5)$. Similarly, \textbf{inserting} a \scycle\ $C =
\cyc(i,i)$ into  $p$ means replacing the vertex $i$ by the cycle
$\cyc$.

\subsection{Block triangular form}\label{ssec:bform}

The \textbf{submatrix} (of $A$) corresponding to  a subset   $\{
i_1, \dots, i_k\}$ of  vertices of the graph $\grph _A$  is
defined to be the $k \times k$ submatrix of $A$ obtained by taking
the $i_1, \dots, i_k$ rows and columns of $A$.

We say $A$ has \textbf{full block triangular form } if   it is
written as
\begin{equation}\label{block1} A = \(  \begin{matrix}  B_1 & B_{1,2} & \dots & B_{1,\tt -1} & B_{1,\tt } \\
 \mzero &  B_2 & \dots & B_{2,\tt -1} & B_{2,\tt }\\ \vdots & \vdots & \ddots  & \vdots
  & \vdots \\ \mzero  &   \dots &  \mzero & B_{\tt -1}
  &  B_{\tt -1,\tt }\\  \mzero&  \dots  &  \mzero  &  \mzero & B_\tt
\end{matrix}\),\end{equation}
 where each diagonal block $B_i$ is an irreducible $n_i \times n_i$
 matrix, $i = 1, \dots, \tt$,
 and each $B_{i,j}$, $j >i$,  is an $n_i \times n_j$ matrix. (Here  we write $\mzero$ for
 the submatrices
 ~$(\rzero)$ in the appropriate positions.) Thus, in this case,  the component graph
 $\igrph _A$
 of $A$ is acyclic and has $\tt$ vertices.

\begin{prop}\label{fullbl} A matrix $A$ is reducible iff it can be put into the following
form (renumbering the indices if necessary):
\begin{equation}\label{block2} A = \(  \begin{matrix}  B_1 & C \\
 (\zero) &  B_2
\end{matrix}\),\end{equation}
 where $n = k + \ell$, $B_1$ is a $k\times k$
matrix, $C$ is a $k\times \ell$ matrix, $B_2$ is an $\ell\times
\ell$ matrix, and $(\zero)$ denotes the zero $\ell\times k$
matrix.

 More generally, any matrix
$A$   can be put into full block triangular form as in
\eqref{block1} (renumbering the indices if necessary), where the
diagonal blocks $B_i$ correspond to the strongly connected
components of $A$.  In this case, $$|A| = |B_1|\cdots|B_\tt|.$$ In
particular, $A$ is nonsingular iff each $B_i$ is nonsingular.
\end{prop}
\begin{proof} Obviously any matrix in the form of \eqref{block2} is reducible. Conversely, suppose that $A$ is
reducible. Take indices $i,j$ with no path from $j$ to $i$. Let
$\mathcal I \subset \{ 1, \dots, n\}$ denote the set of indices of
$\grph_A$ having a path terminating at $i$, and $\mathcal J = \{
1, \dots, n\} \setminus \mathcal I.$ Renumbering indices, we may
assume that $\mathcal I = \{ 1, \dots, \ell\}$ and $\mathcal J =
\{\ell+ 1, \dots, n\}$ for some $1 \le \ell <n$. $A$ is in the
form of \eqref{block2} with respect to this renumbering, and
iterating this procedure puts $A$ in full block triangular form.
\end{proof}

\begin{rem}\label{Jordan0} If $ A$ is in full block triangular form as in \eqref{block1},
 then
 \begin{equation}\label{block3} A^m = \(  \begin{matrix}  B_1^m & ?? & ?? & \dots & ?? \\
 \mzero &  B_2^m & ?? & \dots & ?? \\ \vdots & \vdots & \ddots  & \vdots
  & \vdots \\ \mzero  &   \dots &  \mzero & B_{\tt-1}^m
  & ??\\  \mzero&  \dots  &  \mzero  &  \mzero & B_\tt^m
\end{matrix}\).\end{equation}
 It
follows that $|A^m| = |B_1^m|\cdots|B_\tt^m|$, for any $m \in
\Net$.
\end{rem}

\subsection{Polynomials evaluated on matrices}

Recall that in the supertropical theory we view polynomials as
functions, and polynomials are identified when they define the
same function. Suppose a polynomial $f = \sum_i \a _i \lm ^i$ is a
sum of monomials $\al_i \lm^i$. Let $g = \sum _{i \ne j}\a _i \lm
^i $. The monomial $\a_j \lm ^j$  is \textbf{inessential} in $f$,
iff $f(a) = g(a)$  for every $a\in R$. An inessential monomial $h$
of $f$ is \textbf{quasi-essential} if $f(a) \nucong h(a)$ for some
point $a \in R$. The \textbf{essential part} $\ef$ of a polynomial
 $f = \sum \a _i \lm^i$ is the sum of those monomials
$\a_j \lm ^j$ that are essential.

A polynomial $f \in R[\la]$ is called \textbf{primary} if it has a
unique corner
root(cf.~\cite[Lemma~5.10]{IzhakianRowen2007SuperTropical}), up to
$\nu$-equivalence.

\begin{lem}\label{gen1} If $f \in R[\la]$ is  primary, then $f \nucong \a \sum a^i \la^ {d-i}$, where $\a\la^d$ is
the leading monomial of $f$. If moreover $f \in R[\la]$ is monic
primary with constant term $a^d$, then $(\la + a)^d \lmodg
f.$\end{lem}
\begin{proof} The first assertion is obtained by writing $f$ as a
sum of quasi-essential monomials and observing that corner roots
 are
obtained by comparing adjacent monomials of $f$.

The second assertion follows by expanding $(\la + a)^d = \la^d +
a^d + \sum _{i=1}^{d-1} (a^i )^\nu \la^{d-i}.$
\end{proof}

We say that a matrix $A$ \textbf{satisfies} a polynomial $f \in
R[\la]$ if $f(A) \lmodg (\zero);$ i.e., $f(A)$ is a ghost matrix.
In particular, $A$ satisfies its \textbf{characteristic
polynomial}
$$ f_A := \Det{\la I+ A}= \la ^n + \sum \a_k \la^{n-k},$$ where $\a _k$ is the sum of all
$k$-multicycles in the graph $G_A$ of the matrix $A$,
cf.~\cite[Theorem~5.2]{IzhakianRowen2008Matrices}; those \multi s
of largest $\nu$-value are called the \textbf{dominant $k$-\multi
s} of the characteristic coefficients of $f_A$.

The \textbf{essential characteristic polynomial} is defined to be
the essential part ${f_A}^{\essn}$ of the characteristic
polynomial~$ f_A$, cf.~\cite[Definition
4.9]{IzhakianRowen2007SuperTropical}.  The \textbf{tangible
characteristic polynomial} $\hfA$ of $A$ is defined as $$ \hfA :=
\sum_{k=0}^n \Inu{\a_k} \la^k.$$

Writing $f_A = \la ^n + \sum _{k=1}^{n} \a _k \la ^{n-k}$, we take
\begin{equation}\label{eq:LA} \text{$L(A) : = \{ \ell \ge 1: \root
\ell \of {\a _\ell} \ge_\nu \root k \of {\a _k}$ for each $k \le n
\}$.}
\end{equation}
(There may be several such indices.) In other words, $\ell \in L$
if some $\ell$-multicycle of $A$ has dominating average weight,
either tangible or ghost weight.
\begin{defn}\label{mu01}
We define
\begin{equation}\label{eq:muA}
\mu(A) := \min \{  \ell \ | \ \ell \in L(A)\},
\end{equation}  and call $\a _{\mu}$ the
\textbf{leading characteristic coefficient} of $A$, which we say
 is of \textbf{degree} $\mu(A)$. We denote $\mu (A)$ as $\mu$
if $A$ is understood. We define the \textbf{leading (tangible)
average weight} $\om  := \om  (A)$  to be $$ \om ( A):= \Inu{\root
{\mu} \of {\a_\mu}}.$$
\end{defn}

When $A$ is tangible, $\root {\mu} \of {\a_\mu}$ is the ``maximal
cycle mean'' $\rho_{\max} (A)$ in the sense of \cite{ABG}.

\begin{rem}\label{Jordan00} If $ A$ is in full block triangular form as in \eqref{block1},
 then, by Remark \ref{Jordan0}
the characteristic polynomial of $A^m$ is the product of the
characteristic polynomials of  the $B_i^m$, so many properties of
$A^m$ can be obtained from those of the $B_i^m$.
\end{rem}

We define  the \textbf{(supertropical) trace}
  $\trace{A}$ of the matrix $A = (a_{i,j})$
  to be $$ \trace{A} := \sum_{i=1}^n a_{i,i}.$$

  \begin{rem} Note  that $\a _1 = \trace{A}.$
If   $\a _1 \la^{n-1}$ is an essential or a   quasi-essential
monomial of
  $f_A$, then $\trace {A}$  is the leading characteristic
coefficient  of $A$, and $\mu  = 1$. Furthermore, taking $\bt \in
\tT$ $\nu$-equivalent to $\a_1,$ we know that $\bt$ dominates all
the other supertropical eigenvalues of $A$, and thus in view of
\cite[Theorem~7.10]{IzhakianRowen2008Matrices} is the eigenvalue
of highest weight.
\end{rem}

\begin{example}  The characteristic polynomial of any $n \times n$
quasi-identity matrix $I_ \tG$ is $$\la^n + \sum _{i =1}^{n-1}
\rone^\nu \la^{n-i} + \rone,$$ since any \scycle\ contributing a
larger characteristic coefficient could be completed with $\rone$
along the diagonal to a dominating ghost contribution to $|I_
\tG|,$ contrary to the fact that $|I_ \tG| = \rone.$ (This
argument is implicit in
\cite[Remark~4.2]{IzhakianRowen2008Matrices}.) Hence, $\mu = 1 ,$
and the leading characteristic coefficient of $I_ \tG$ is $\trace
{I_ \tG}$, which is
 $\rone ^\nu$. \end{example}


\begin{lem}\label{power0} If  $f = \sum
\a_i \la^i,$ then $f(A)^m  \lmodg  g(A^m),$ where $g = \sum \a_i^m
\la^i.$
\end{lem}
\begin{proof} By Proposition \ref{Frob},  $$f(A)^m = \bigg (\sum \a_i
A^i \bigg)^m \lmodg \sum \a_i^m A^{im} = g(A^m),$$ again by
Proposition \ref{Frob}.
\end{proof}

\begin{lem}\label{char3} If $A$ is as in Remark~\ref{Jordan0}, then $f_{A^m} =
f_{B_1^m}\cdots f_{B_\tt^m}.$
\end{lem}
\begin{proof} $f_{A^m} = |\la I + A^m|,$ which is the product of
the determinants of the diagonal blocks of $\lm I + A^m$, i.e.,
the $|\la I + B_j^m|$.
\end{proof}

\section{Powers of matrices}

We would like to study powers of a matrix $A \in M_n(R)$, where
$R$ is a supertropical domain, in terms of properties of its
characteristic polynomial $f_A$. Certain properties can be had
quite easily.
\begin{lem}\label{power} If $A$ satisfies the polynomial $f = \sum
\a_i \la^i,$ then $A^m$ satisfies the polynomial $\sum \a_i^m
\la^{i}.$\end{lem}
\begin{proof} This is a special case of Lemma~\ref{power0}.
\end{proof}

\begin{thm}\label{char1} If the characteristic polynomial $f_A  =
\sum_{i=0}^n \a_i \la^i,$ then  $f_{A^m} \lmodg  \sum_{i=0}^n
\a_i^m \la^i,$ for any $m$.
\end{thm}
\begin{proof} Application of Lemma~\ref{power} to  \cite[Theorem~5.2]{IzhakianRowen2008Matrices}.
\end{proof}


Idempotent matrices need not have diagonal entries $
\nu$-equivalent to $ \rone$; for example the matrix
$$ A =\( \begin{matrix}
  -1 & -2 \\
  1 & 0 \\
\end{matrix}\)$$
(in logarithmic notation) is  idempotent, but also singular; i.e.,
$|A| = (-1)^\nu$. Later, cf.~Lemma \ref{quasired}, we see that
this is impossible for nonsingular matrices.

\begin{example}\label{badex} Suppose $$A =   \( \begin{matrix} 0 & 0 \\
1 & 2
\end{matrix}\)$$ in logarithmic notation.  Then $\trace{A} = 2$ and $\Det{A} = 2,$ so $$f_A =
\la^2 + 2\la + 2 = (\la +2)(\la + 0).$$   Note that $\mu(A) = 1 $
and $\a_1 = 2.$

 On the other hand,
$$A^2 =   \( \begin{matrix} 1 & 2 \\
3 & 4
\end{matrix}\);$$
$\trace{A^2} = 4$ and $\Det{A}^2 = 5^\nu,$ so $$f_{A^2} = \la^2 +
4\la + 5^\nu.$$ By Lemma~\ref{power}, $A^2$ also satisfies the
polynomial $\la^2 + 4\la +4.$ For $A^4$ we then have  $$A^4 =   \( \begin{matrix}5 & 6 \\
7 & 8
\end{matrix}\) = 4A^2,$$ and in general $A^{2k} = (A^4)^k =  4^k A^2 =
2^{2k}A^2.$

\end{example}

\subsection{Ghostpotent matrices}

In  Example~\ref{power1}, no power of $A$ is ghost, and we would
like to explore such a phenomenon.

\begin{defn} The matrix $\GP$ is \textbf{ghostpotent} if  $A^m \in M_n(\tGz)$, i.e., $A^m \lmodg (\zero)$,  for some
$m >0.$ The least such $m$ is called the \textbf{ghost index} of
the ghostpotent matrix $\GP$.
\end{defn}

It easy to check that if $A$ is ghost (i.e., has ghost index 1) so
is $A B$ for any $B \in M_n(R)$, and therefore if $A$ is
ghostpotent with ghost index $m$ then $A^k$  is ghost for any $k
\geq m$.
 However, in contrast to the classical
theory, the ghost index of an $n \times n$ ghostpotent matrix need
not be $\leq n$.  Although in cf.~\cite[Theorem
3.4]{IzhakianKnebuschRowen2009Linear} it is shown that the product
of two nonsingular matrices cannot be ghost, a ghostpotent matrix
can still  be nonsingular.

\begin{example}\label{ex:ghostpotent} The nonsingular
matrix $\GP =   \( \begin{matrix} \rzero & \rone \\
\rone & \rone
\end{matrix}\)$ is ghostpotent, for which $\GP^2 =  \( \begin{matrix} \rone & \rone \\
\rone & \rone^\nu
\end{matrix}\)$ is singular, $\GP^3 =  \( \begin{matrix} \rone & \rone^\nu \\
\rone^\nu & \rone^\nu
\end{matrix}\)$, and only for $m =4$ do we obtain the ghost matrix
$\GP^4 =  \( \begin{matrix} \rone^\nu & \rone^\nu \\
\rone^\nu & \rone^\nu
\end{matrix}\)$ (which is $\nu$-equivalent to $A^2$). Note here that  $f_A = \la^2 + \la + \rone,$ and $\mu(A) = 1
,$ even though the monomial  $\la $ is not essential, but only
quasi-essential.

\end{example}

From this example, we see that the image of the action of a
ghostpotent matrix $\GP$ on a vector space can be a thick subspace
\cite[\S 5.5]{IzhakianKnebuschRowen2009Linear}, for $\GP$ can be
nonsingular.

\begin{example}\label{ex:ghostpotent1} Let  $\GP  =  \( \begin{matrix} \rone   & a  \\
b  & \rone
\end{matrix}\)$.
\begin{enumerate} \eroman
    \item
When $ab < _{\nu} \rone,$ the
matrix $\GP ^2 =   \( \begin{matrix}   \rone & a^\nu \\
b^\nu & \rone
\end{matrix}\)$ is nonsingular idempotent (in fact a quasi-identity matrix), and thus not ghostpotent.

\item
When $ab  \nucong \rone,$ then  $\GP^2 =  \( \begin{matrix} \rone ^\nu & a ^\nu\\
b^\nu & \rone^\nu
\end{matrix}\)$, which is already ghost. \pSkip

\item  When $ab > _{\nu} \rone,$ then
 $\GP^2 =  \( \begin{matrix} ab& a ^\nu\\
b^\nu & ab
\end{matrix}\)$ and  $\GP^4 = ab  \GP^2 $. Thus, in this case,  $\GP$ is ghostpotent iff $a$ or $b$ is a
ghost.
\end{enumerate}

\end{example}

 Here is an instance where one does have a strong
bound on the ghost index.

\begin{prop}
If the only supertropical eigenvalue of $\GP$ is $\rzero$,  then
$\GP^{n} \lmodg (\rzero)$.\end{prop}
\begin{proof}  The characteristic polynomial $f_A$ of $A$ cannot
have any roots other than $\rzero$, and thus $f_A = \la^n.$ Hence,
  $A^n
 \in M_n(\tGz)$ by  \cite[Theorem~5.2]{IzhakianRowen2008Matrices}. \end{proof}

The following result enables us to reduce ghostpotence to
irreducible matrices. (We write $\mzero$ for $(\rzero)$.)

\begin{lem}\label{ghostpot}   If $$N =  \(  \begin{matrix}  N_1 & ? & ? & \dots & ? \\
 \mzero &  N_2 & ? & \dots & ? \\ \vdots & \vdots & \ddots  & \cdots
  & \vdots \\ \mzero  &   \dots &  \mzero & N_{\tt-1}
  & ?\\  \mzero&  \dots  &  \mzero  &  \mzero & N_\tt
\end{matrix}\),$$  where $N_i^{m_i} $  is ghost, then $N^{\tt m}\in
M_{n}(\tGz)$ for any $m \ge  \max \{ m_1, \dots, m_\tt \}.$
\end{lem}
\begin{proof}
$$N^m = \(  \begin{matrix}  N_1^m & ?? & ?? & \dots & ?? \\
 \mzero &  N_2^m & ?? & \dots & ?? \\ \vdots & \vdots & \ddots  & \cdots
  & \vdots \\ \mzero  &   \dots &  \mzero & N_{\tt-1}^m
  & ??\\  \mzero&  \dots  &  \mzero  &  \mzero & N_\tt^m
\end{matrix}\),$$ which is ghost on the diagonal blocks, and the $\tt$ power of this matrix makes everything ghost. \end{proof}

Theorem~\ref{ghostpot1} will give us a complete determination of
ghostpotent matrices.

\subsection{Computing powers of matrices}

Our next concern is to compute powers of a matrix $A$, in order to
determine whether some power $A^m$ of $A$ is a singular matrix (or
even ghost), and, if so, to determine the minimal such~$m$.  There
is no bound on the power we might need to get a singular matrix.

\begin{example}\label{higho} $ $

 Let   $A = \(  \begin{matrix} a & \rone \\
\rone & b
\end{matrix}\)$, where $ a >_\nu b >_\nu \rone$ are tangible, and thus $A$ is nonsingular.   \text{Then } \
$$ A^k =  \left( \begin{matrix} a^{k} & a^{k-1} \\
a^{k-1}  & a^{k-2} +b^k \end{matrix}\right)= a^{k-1} \left( \begin{matrix} a  &\rone\\
\rone  & a^{-1} +\frac{b^k}{a^{k-1}} \end{matrix}\right).$$ We are
interested in the lower right-hand term. This is
$\frac{b^k}{a^{k-1}}$ (and thus tangible) so long as it dominates
$a^{-1}$. On the other hand, if $b^k \nucong a^{k-2}$ then $$ A^k =   a^{k-1} \left( \begin{matrix} a  & \rone\\
\rone  & (a^{-1})^\nu   \end{matrix}\right),$$ and if   $b^k <_\nu a^{k-2}$ then $$ A^k =   a^{k-1} \left( \begin{matrix} a  & \rone\\
\rone  & a^{-1}    \end{matrix}\right),$$ both of which are
singular.

In other words, taking tangible $c> _\nu \rone,$ if $a = c^{k}$
and $b = c^{k-2}$, then $A^{k-1}$ is nonsingular whereas $A^k$ is
singular.
\end{example}

Nevertheless, one can get information from the  leading
characteristic coefficient. We write $\mu$ for $\mu(A)$,
cf.~Definition~\ref{mu01}.
 As   noted following
\cite[Definition~5.1]{IzhakianRowen2008Matrices}, the coefficient
$\a_\ell$ of the characteristic  polynomial~$f_A$ is the sum of
the weights of the $\ell$-multicycles in the graph $\grph_A$.
Thus, the leading characteristic coefficient~$\a _\mu$ of~$A$ is
the sum of the weights of the multicycles of length $\mu$ (also
having maximal average weight, whose tangible value is denoted as
$\om $) in the weighted digraph $\grph_A$ of $A$. Accordingly, let
us explore the multicycles contributing to $\a _\mu.$

\begin{lem}  Any multicycle contributing to
the  leading characteristic coefficient must be a
\scycle.\end{lem}

\begin{proof}  If some dominant such
multicycle were not a \scycle, it could be subdivided into smaller
disjoint \scycle s, at least one of which would have average
weight
  $ \geq_\nu \om $ (and a shorter length) and thus which would give a leading
characteristic coefficient of lower degree, contrary to the
definition of leading characteristic coefficient.
\end{proof}

Thus, we can focus on  \scycle s.

\begin{defn} A \textbf{leading \scycle} is a \scycle\  whose average weight is
$\nu$-equivalent to $\om $. A \textbf{leading $\ell$-\scycle} is a
leading \scycle \ of length $\ell$. (In particular, $\a _\ell$
equals the sum of the weights of the leading $\ell_i$-\scycle s
with $\sum \ell_i = \ell$.) The number of leading $\ell$-\scycle s
is denoted $\tau_\ell$.
Given an index
 $i$, we
 define its
\textbf{depth} $\rho_{i}$ to
 be the number of leading \scycle s of $\grph_A$ containing $i$.
 \end{defn}

In view of~\eqref{eq:muA}, the length $\ell$ of a leading
$\ell$-\scycle\ must be between $\mu$ and
  the least degree of monomials in the essential characteristic
  polynomial.
  In Example~\ref{ex:ghostpotent} there are leading
  \scycle s of length both $1$ and~$2$.
Note that
  $\rho_i = \mu =1$ where  $\tau_i = 1,$ for $i = 1,2$.

\begin{example} Any quasi-identity matrix satisfies $\mu = \rho_i  =1$ for each $i$, whereas $\tau_\mu = n$.  \end{example}

\begin{lem}  If the  leading characteristic coefficient $\a
_\mu$ of $A$ is tangible, then $\tau_\mu = 1$.\end{lem}
\begin{proof}  Otherwise, $\a
_\mu$ would be the sum  of several  $\nu$-equivalent weights, and
thus must be ghost.
\end{proof}

\begin{lem}\label{mu1} $\mu(A^\mu)=1$, for any matrix $A$.
\end{lem}
\begin{proof} Let $A^\mu = (b_{i,j}),$ and  let $\cyc_\mu = C(i,i)$ be a leading $\mu$-scycle \ of $A$.
Then, $ b_{i,i} = w(C_\mu)$, and $(i,i)$ is a  $1$-multicycle of
$A^\mu$, which  is comprised of just one \scycle \ of length $1$,
and is clearly a leading \scycle \ of~$A^\mu$.
\end{proof}

\begin{defn}\label{core} A \scycle\ is \textbf{core-admissible} if each of its vertices has
depth 1; i.e., it is disjoint from each other leading \scycle.
 The \textbf{core} of an irreducible matrix
$A$, written $\coreA$, is the \multi\ comprised of the union of
all core-admissible leading \scycle s. The \textbf{tangible core}
of $A$, written $\tancoreA$, is the \multi\ comprised of the union
of all tangible core-admissible leading \scycle s.

\end{defn}

Thus, a leading \scycle\ is part of the core iff its vertex set is
disjoint from all other leading \scycle s in~$A$. Note that
$\tcoreA \subseteq \coreA$,  and also note that
$\coreA$  and  $\tcoreA$ can be empty; for example the $\core$ of  $A = \left(%
\begin{array}{cc}
  \rone & \rone \\
  \rone & \rone \\
\end{array}%
\right)$ is empty.

We write $(A)_{\core}$ (resp. $(A)_{\tancore}$) to denote the
submatrix of
 $A$ comprised of the rows and
  columns corresponding to the indices of $\mV(\coreA)$ (resp.
  $\mV(\tancoreA)$).

  The idea
behind the core is that the vertex set of $\tancoreA$ is comprised
precisely of those vertices of leading \scycle s which contribute
tangible weights to high powers of $A$. The other vertices of
leading \scycle s provide the ``ghost part,'' so we also consider
them.

 \begin{defn} The \textbf{anti-tangible-core},  denoted $\acore (A)$,
  is the \multi\ comprised of the union of those
leading \scycle s which are not in $\tancoreA$. We write
$(A)_{\acore}$ to denote the submatrix of
 $A$ comprised of the rows and
  columns corresponding to $\Vacore.$
\end{defn}

The anti-tcore is the set of vertices $i$   for which the $(i,i)$
entry of high powers of $A$ become ghosts. (Note that $\coreA \cap
\acoreA$ could be nonempty, when $A$ has a core-admissible cycle
with ghost weight.)

 The leading
\scycle s appear in the following basic computation.

\begin{rem}\label{extract} The $(i,j)$ entry of a given power $A^m$ of $A$ is the
sum of weights of all paths of length~$m$ from $i$ to $j$ in
$\grph_A.$ By the pigeonhole principle,  any   path $p$ of length
$\ge n$ contains an \scycle\ $\cyc$ (since some vertex must
repeat, so $\weight {p} = \weight{\cyc} \weight{p'}$ where $p'$ is
the path obtained by deleting the \scycle\ $\cyc$ from $p$.
Continuing in this way enables us to write $\weight {p} $ as the
product of weights of \scycle s times a simple path of length $<n$
from $i$ to $j$.

If $i=j$, then $\weight{p}$ can thereby be written as a product of
weights of \scycle  s. 
 Since, by
definition, the average weight of each \scycle\ must have
$\nu$-value at most $\om  $, the weight of $p$ is at most $w^m,$
the maximum being attained when all the \scycle s are leading
\scycle s.

If $i\ne j,$ one has to consider all paths of length $<n$ from $i$
to $j$ and take the maximum weight in conjunction with those of
the \scycle s; this is more complicated, but we only will need
certain instances.\end{rem}

\begin{example} Taking $$A = \(  \begin{matrix} \rzero & \rzero & \rone & \rone & \rzero\\
\rzero & \rzero & \rzero & \rzero & \rone \\ \rzero & \rone  &
\rzero & \rzero & \rzero\\ \rzero & \rone  & \rzero & \rzero &
\rzero\\ \rone  & \rzero & \rzero & \rzero & \rzero
\end{matrix}\),$$ we have $\tV(\core(A)) = \emptyset$ since the two leading \scycle s
$(1,3,2,5)$ and $(1,4,2,5)$ intersect at the vertex 2. But $$A^2 =  \(  \begin{matrix} \rzero & \rone^\nu  & \rzero  & \rzero  & \rzero\\
\rone  & \rzero & \rzero & \rzero & \rzero \\ \rzero & \rzero &
\rzero & \rzero & \rone \\ \rzero & \rzero & \rzero & \rzero &
\rone \\ \rzero & \rzero &  \rone  &  \rone  &  \rzero
\end{matrix}\),$$  and so $\tV(\core(A^2)) = \{ 1, 2\}$ whereas
still $\tV(\tancore(A^2)) = \emptyset $ since the leading two
tangible \scycle s $(3,5)$ and $(4,5)$ intersect.
\end{example}

Thus, our next result is sharp. By $\cyc^k$ we mean the
concatenation of $\cyc$ taken $k$ times; for example,
$$(1,3,5,2,1)^3 = (1,3,5,2,1,3,5,2,1,3,5,2,1).$$
In the reverse direction, we can decompose a cycle into a union of
other cycles by ``skipping'' vertices. For example, skipping two
vertices each times decomposes $(1,3,5,2,1,4,5,2,1,3,4,2,1)$ into
the three cycles $(1,2,5,3,1)$, $(3,1,2,4,3)$, and $(5,4,1,2,5)$.

\begin{prop}\label{rmk:core1} $ $
\begin{enumerate} \eroman
    \item
If $\cyc$ is a leading \scycle\ in $\grph _{A}$ having average
weight $\om $, then for any $k\in \Net$,  the cycle $\cyc ^k$
decomposes into a union of leading \scycle s for $\grph _{A^k}$
each having average weight $\om ^k$, where we take every $k$
vertex, as indicated in the proof. \pSkip

\item    $\tV(\core(A)) \subseteq \tV(\core(A^k))$ and
$\tV(\tancore(A)) \subseteq \tV(\tancore(A^k))$, for any $k \in
\Net$. \pSkip

\item  $\tV(\tancore(A)) = \tV(\tancore(A^k))$, for any $k \in
\Net$.
\end{enumerate}
\end{prop}
\begin{proof} (i) Let $\cyc = (i_1, \dots, i_\mu).$ Then $\cyc ^k = (i_1, \dots, i_\mu)\cdots (i_1, \dots, i_\mu),$
which in $\grph _{A^k}$ appears as
\begin{equation}\label{cycpow} (i_1, i_{1+k}, \dots)(i_{j_2},
i_{j_2+k}, \dots)\cdots.\end{equation} But $\weight{\cyc'} \le_\nu
\om^k$ in $\grph _{A^k}$, since otherwise we could extract a cycle
in $\grph _{A}$ of weight $>_\nu \om,$ contrary to hypothesis. It
follows that each cycle in \eqref{cycpow} has average weight
$\om^k$.

(ii) If an index  $i \in \tV(\core(A)) $ appears in a single
leading \scycle\ $\cyc$ of $\grph _{A}$ then it appears in the
corresponding leading \scycle\ of $\grph _{A^k}$, according to
(i), and cannot appear in another one, since otherwise we could
extract another leading \scycle\ of $\grph _{A}$ whose vertex set
intersects that of $\cyc$, contrary to hypothesis. Furthermore, if
$\cyc$ has tangible weight, then so does $\cyc ^k$.

(iii)  The same argument as in Remark~\ref{extract} shows that,
for any $m$ and any leading \scycle\ $\cyc'$ of $A^k$, we can
extract leading \scycle s of $A$ until we obtain the empty set. If
their vertices are all in $\tV(\core(A)),$ then clearly
$\tV(\cyc') \subseteq \tV(\core(A^k)),$ as desired.
\end{proof}

\begin{lem}\label{dupl} There is a number $m = m(n, \mu)$ such
that,
 for  any path  $p$  of length $ \ge m$ from
vertices $i$ to $j$ in $\grph _{A}$ having maximal weight and for
which its vertex set $\tV(p)$ intersects $\tV(\coreA)$, $p$
contains a set of leading \scycle s of total length a multiple
$\ell \mu$ for some $\ell \le n$.

When $\mu =1,$ we can take $m = 2n-1.$ In general, we can take
 $m = (\binom {n+1}2 -\mu)(\mu -1) +2(n-1) +1.$
\end{lem}
\begin{proof} Take a leading \scycle\ $\cyc_\mu$   of length $\mu$,
with  $k \in \tV(\cyc_\mu) \cap \tV(p)$. We are done (with $\ell =
1$) if $p$ contains a \scycle\ $\cyc$ of length $\mu$, since we
could delete $\cyc$ and insert $\cyc_\mu$ at vertex $k$, contrary
to the hypothesis of maximal weight unless $\cyc$ itself is a
leading \scycle. Thus, we may assume that $p$ contains no \scycle\
of length $\mu$. Next, any path of length $\ge n$ starting from
$k$ contains a \scycle\ that could be extracted, so if the path
has length $\ge s+(n-1)$ we could extract  \scycle s of total
length at least $s$; likewise for any path ending at $k$.

Thus, for $\mu =1,$ we could take $s=1$ and $m = (n-1)+(n-1)+1 =
2n-1.$ Then we  are be able to extract a \scycle\ of some length
$\ell \le n,$ which we could replace by $\ell $ copies of
$\cyc_\mu$ each of which is a leading \scycle.

In general, if $p$ has length at least $(\binom {n+1}2 -\mu)(\mu
-1) +2(n-1) +1,$ using the pigeonhole principle, we are be able to
extract $\mu$ \scycle s of some length $\ell \le n,$ which we
could replace by $\ell $ copies of $\cyc_\mu$, so by the same
argument, each of these \scycle s has average weight $\om $, and
thus is a leading \scycle.
\end{proof}

Let $\tlmu $ denote the least common multiple of the lengths of
leading  \scycle s of $A$; i.e.,  every $\ell \in L(A)$ divides $
\tlmu ,$
 cf.~\eqref{eq:LA}. In particular, $\mu$ divides ~$\tlmu $.

\begin{prop}\label{prop:4.20}

Suppose $\cyc$ is an \scycle\ in $\coreA$, with $i \in \mV(\cyc)$,
of weight $w_C = w(C)$ and of length $\ell = \ell (C)$.
 Then, for any $k$, the $(i,i)$-diagonal
entry of $A^{k\tlmu }$ is $(w_C)^{k \tlmu / \ell}$, which is
$\nu$-equivalent to~$\om ^{k \tlmu }$. Furthermore, assuming that
$\coreA$ is nonempty,
$$|(A^{k\tlmu })_{\core}|= \prod_{C \subset  \coreA} (w_C)^{k
\tlmu / \ell(C)},$$ i.e.,  $|(A^{k\tlmu })_{\core}| \nucong \om
^{k s\tlmu } $, where $s= \#(\tV(\coreA))$.\end{prop}
\begin{proof} There is only one dominant term in the
$(i,i)$ entry of $A^{k\tlmu }$, which comes from repeating the
single $\ell$-leading \scycle\ $\cyc(i,i)$ containing $i$
(starting at position $i$)  $k\tlmu / \ell$ times. Since $C(i,i)$
is a core-admissible leading \scycle, $\sqrt[\ell]{C(i,i)} \nucong
\om$. This proves the first assertion.

First assume for simplicity that $\core(A)$  is comprised of a
single \scycle. Any other contribution to the determinant of
$(A^{k\tlmu })_{\core}$ would also come from a multicycle, and
thus from a power of $\cyc$, by assumption a unique leading cycle,
which must then be the same product along the diagonal. Thus, the
single leading multicycle of $A^{k\tlmu }$ is the one along the
 diagonal, which yields the determinant.

The same argument applies simultaneously to each core-admissible
\scycle. Namely, any dominant term along the diagonal
 must occur from repeating the same scycle, since the leading scycles are presumed disjoint,
 and again the single leading multicycle of $A^{k\tlmu }$ is  the
 diagonal.
\end{proof}

\begin{cor}  Assuming that
$\tcoreA$ is nonempty, $|(A^{k\tlmu })_{\tcore}|= \om^{k s
\tlmu},$ where $s= \#(\tV(\tcoreA))$.
\end{cor}

Recall that the rank of a matrix is the maximal number of
tropically independent rows (or columns), which is the same as the
maximal size of a nonsingular submatrix,
cf.~\cite{IzhakianRowen2009TropicalRank}.

\begin{cor} The rank of every power of $A$ is at least $\# (\tV(\tcoreA))$.
 In particular, if $\tcoreA$ is nonempty, then $A$ is
not ghostpotent.\end{cor}

\begin{cor} Suppose that the leading \scycle s of $A$ are disjoint. Then, for any $k$, and any vertex $i$
of a leading $\ell$-\scycle~$C$,  the $(i,i)$-diagonal entry of
$A^{k \ell}$ is $w(C)^k$, which is $\nu$-equivalent to
  $\om ^{k \ell}$. \end{cor}

%
%
%
%
%
%

\begin{example} Let
$$ A = \left(%
\begin{array}{ccc}
  0 & 2 & 4 \\
  4 & 0 & -1 \\
  1 & 0 & 3^\nu \\
\end{array}%
\right). \quad \text{Then} \quad  A^2 = \left(%
\begin{array}{ccc}
  6 & 4 & 7^\nu \\
  4^\nu & 6 & 8 \\
  4^\nu & 3^\nu & 6^\nu \\
\end{array}%
\right) \quad  \text{and} \quad A^4 = \left(%
\begin{array}{ccc}
  12 & 10^\nu & 13^\nu \\
  12^\nu & 12 & 14^\nu \\
  10^\nu & 9^\nu & 12^\nu \\
\end{array}%
\right). $$ For the matrix $A$ we have $\mu =1$, $\al_\mu = 3^\nu$
and thus $\om = 3$.  Moreover, $\tlmu = 2$ and in this matrix
$\coreA = \tcoreA = A$; thus $A^{2 \tlmu} = A^4$ and  hence
$$  A^4 = \left(%
\begin{array}{ccc}
  12 & 10^\nu & 13^\nu \\
  12^\nu & 12 & 14^\nu \\
  10^\nu & 9^\nu & 12^\nu \\
\end{array}%
\right)
 = 12 \left(%
\begin{array}{ccc}
  0 & -2^\nu & 1^\nu \\
  0^\nu & 0 & 2^\nu \\
  -2^\nu & -3^\nu & 0^\nu \\
\end{array}%
\right), $$ where the matrix on the right is idempotent.
\end{example}

\subsection{Semi-idempotent matrices}

\begin{defn} A matrix $A\in M_n(R)$ is \textbf{semi-idempotent} if $A^2 =
\bt  A$ for some tangible $\bt = \bt(A)$ in $R.$ We call $\bt
(A)$ the \textbf{semi-idempotent coefficient} of $A$.

\end{defn}

\begin{rem}\label{prop:semiIdem} If $A$ is a  semi-idempotent matrix, then
$A^k = A^{k-2} A^2 =  A^{k-2}  \bt A = \bt A^{k-1}  = \cdots =
\bt^{k-1} A,$
 and any power
  of $A$ is semi-idempotent.
\end{rem}

\begin{lem}\label{semiidem} If $A$ is   semi-idempotent, then $\mu (A) = 1.$
\end{lem}
\begin{proof} By Remark \ref{prop:semiIdem}, $A^\mu =
\bt^{\mu-1}A$, for any $\mu \geq 1$. Clearly,  $\mu( \al A) =
\mu(A)$ for any $a \in R$.   Then, $\mu( A) = \mu(A^\mu)=1,$
 by Lemma~\ref{mu1}. \end{proof}
 The next result ties this concept in with
\cite{IzhakianRowen2008Matrices, IzhakianRowen2009Equations}.
\begin{lem}\label{quasired} If  $A$ is  a nonsingular semi-idempotent matrix, then $|A|
= \bt $ and  ${ {\bt}}^{-1} A$ is a quasi-identity
matrix.\end{lem}
\begin{proof} Clearly ${ {\bt}^{-1}} A$ is nonsingular semi-idempotent, since $\bt^{-1}$ is tangible, and its
determinant is tangible. So $ |A^2| = |\bt  A| = \bt |A|$ is
tangible,
 implying $$\bt  |A| = |A^2| = |A|^2,$$ by Equation \eqref{matmul}, and thus
$\bt  = |A|$. Hence, $|{ {\bt}^{-1}} A| = { {\bt}^{-1}} |A| =
\rone$, implying the matrix ${ {\bt}^{-1}} A$ is idempotent, and
thus is  a quasi-identity matrix by
Proposition~\ref{rmk:quasisingular0}.
\end{proof}

\begin{example} The matrix $$ A = \(\begin{array}{cc}
         1 & 2 \\
           3 & 4 \end{array}\) $$
           satisfies $A^{2} = 4A$, so $A$ is a singular semi-idempotent matrix.    \end{example}

\begin{thm}\label{corepower} $ $
\begin{enumerate}\eroman
    \item For any matrix $A$ with nonempty
core, the submatrix
  $(A^{ m\tlmu  })_{\core}$ is
semi-idempotent for some power $m$, with semi-idempotent
coefficient $\bt ((A^{ m\tlmu })_{\core}) = \om ^{m \tlmu }$.
\pSkip

 \item For any matrix $A$ with nonempty
tcore, the submatrix
  $(A^{ m\tlmu  })_{\tcore}$ is
semi-idempotent for some power $m$, with semi-idempotent
coefficient $\bt ((A^{ m\tlmu })_{\tcore}) = \om ^{m \tlmu }$, and
hence $(\om ^{m \tlmu })^{-1}(A^{ m\tlmu  })_{\tcore}$ is a
quasi-identity matrix.
\end{enumerate}

\end{thm}
\begin{proof} (i) The $(i,j)$ entry $b_{i,j}$ of $B = (A^{ m\tlmu  })_{\core}$ is obtained from
some path $p$ of length $m\tlmu $ from $i$ to $j$ in the digraph
$G_A$, from which we extract as many \scycle s as possible;
cf.~Remark \ref{extract}, to arrive at some simple path $p'$ from
$i$ to $j$ without \scycle s; thus $p'$ has length $ \leq n$, and
length $<n$ when $i\neq j$. $p$, as well as  each of whose \scycle
s, has a ghost weight iff it has a ghost edge. Furthermore, by
Lemma~\ref{dupl}, given a leading  $\ell$-\scycle\ $\cyc$ in
$G_A$, we could replace $\cyc ^{ \tlmu }$ by $\om^{\ell \tlmu }$
or $(\om^{\ell \tlmu })^\nu$, depending whether $\cyc$ is ghost or
not, without decreasing the $\nu$-value of $ b_{i,j}$; hence we
may assume that it is possible to extract at most $\tlmu  -1$
\scycle s of length $\ell \ne \mu$ from $p$, for each $\ell$.

Working backwards, we write $p'_\ell(i,j)$ for the path of length
$\ell$  from $i$ to $j$ in $G_A$ having maximal weight. There are
only finitely many maximal possibilities for multiplying this by
weights of \scycle s of length $\ne \mu$, so at some stage, taking
higher powers $m$ only entails multiplying by $\om^\mu$. Doing
this for each pair $(i,j)$ yields the theorem. \pSkip

 (ii) The same argument as in (i), noting that the diagonal now is tangible.
\end{proof}

 In case $A$ is an irreducible matrix, it is known that
  $A^{ m\tlmu  +1} = \om  ^{ m\tlmu } A$
 over the max-plus algebra; cf.~\cite[\S 25.4,
Fact 2(b)]{ABG}, where $m$ is called the \textbf{cyclicity}. This
does not quite hold in the supertropical theory, because of the
difficulty that taking powers of a matrix might change tangible
terms to ghost. Thus, we must settle for the following result.

\begin{cor}\label{corepower2} In case $A$ is irreducible, $A^{ m\tlmu  +1} \nucong \om  ^{
m\tlmu } A$ and $A^{k m\tlmu  +1} = \om  ^{(k-1) m\tlmu } A^{
m\tlmu +1}$ for all $k>1. $
\end{cor}
\begin{proof} Take $m$ as
in Theorem~\ref{corepower}(i). Let $c_{i,j}$ denote the $(i,j)$
entry
 of $A^{ m\tlmu  +1}$.
 To prove the first assertion, we need
 to  show
that $c_{i,j} = \om  ^{ m\tlmu } a_{i,j}.$ We take a maximal path
$p$ from $i$ to $j$ in the graph  of $A^{ m\tlmu +1}$.
 Let $b_{i,j}$ denote
the $(i,j)$ entry of $A^{ m\tlmu }$, the weight of the path $p$.
Then $c_{i,j} \le_\nu \om ^{ m\tlmu } a_{i,j}$ by
Proposition~\ref{prop:4.20}, since  every cycle of length $\mu$
must have weight  $\le \om ^{ m\tlmu }$.
 On the other hand, Lemma~\ref{dupl} gives us leading \scycle s of total length $\ell
\mu$. This yields $\om  ^{ m\tlmu } a_{i,j} \le \om
^{\ell\mu}b_{i,j} \le_\nu c_{i,j};$ indeed any cycle of length
$\mu$ in $p$ must have weight  $\om ^{ m\tlmu }$ (since otherwise
it could be replaced by $\cyc$, thereby providing a path of weight
$>w(p)$, a contradiction). Thus, equality holds and we proved the
first assertion.

The second assertion follows, by taking one more pass through the
\scycle s, as illustrated in Example~\ref{ex:ghostpotent}.
\end{proof}

 We obtain a generalization in the
non-irreducible case, in Theorem~\ref{stabil}.

\begin{cor}\label{semiid} Any matrix in full block triangular form has a power such that each diagonal
block is  semi-idempotent.
\end{cor}
\begin{proof} Apply Theorem~\ref{corepower}(ii) to each diagonal
block.
\end{proof}

\begin{cor}\label{cor:core}  $(A^{ m\tlmu })_{\core}$ is $\om^{m\tlmu}$
times an idempotent matrix $J_\tG \lmod I_\tG$, where $\tG$ is a
quasi-identity matrix, for some $m$.\end{cor}
\begin{proof} Dividing out by
 $\om^{m\tlmu}$,  we take a suitable power and may assume that
$(A)_{\core}$ is idempotent. But we can also replace each of the
diagonal entries of $(A)_{\core}$ by $ \rone$ or $\rone^\nu$, and
then are done by~Corollary~\ref{corepower2}.
\end{proof}

\begin{cor}  $(A^{ m\tlmu })_{\tancore}$ is $\om^{m\tlmu}$ times a quasi-identity matrix, for some
$m$.\end{cor}
\begin{proof} Consequence of Corollary \ref{cor:core}.
\end{proof}

\begin{example} The matrix $$ A = \(\begin{array}{cc}
          \rzero & \rone  \\
           \rone & \rzero \end{array}\) $$ satisfies
           $A^{2n+1} = A$ but  $A^{2n-1}A= A^{2n} = I$  for each
           $n$.
           \end{example}

Note that in Example~\ref{badex}, $A^2$ is semi-idempotent and not
ghostpotent, but singular. Let us now consider \scycle s  in the
anti-tcore.

\begin{prop}\label{anti1}  Suppose $\rho_i >1$. Then, for any $k\ge 2$, the $(i,i)$-diagonal
entry of $(A^{k\tlmu })_{\acore}$ is $(\om ^{k \tlmu })^\nu$.
Furthermore, when $\acoreA$ is nonempty,  $$|(A^{k\tlmu
})_{\acore}|= (\om ^{s k \tlmu })^\nu,$$ where $ s = \#(\Vacore).$
\end{prop}
\begin{proof} There are at least two dominant terms in the
$(i,i)$ entry of $(A^{k\tlmu })_{\acore}$, which come from
exchanging two leading \scycle s $\cyc(i,i)$ containing $i$
(starting at position $i$), or a dominant ghost term which comes
from a leading ghost \scycle \ containing $i$.
\end{proof}

\begin{thm}\label{singp} If $\acore (A)$ is nonempty, then some power of $A$ is singular.
 If $A$ is irreducible and $\tancoreA$ is empty, then $A$ is ghostpotent.

More generally, for $A$ irreducible, there is a power of $A$ such
that the $(i,j)$ entry of $A$ is ghost unless $i,j \in
\tV(\tancore(A)).$
\end{thm}
\begin{proof} The diagonal elements from
$(A^  {m\tlmu })_ {\tancore}$ and $(A^  {m\tlmu })_{\acore}$ occur
in the multicycle determining $|A^  {m\tlmu }|$ for large $m$,
yielding the first assertion. To prove the last assertion (which
implies the second assertion), we need to show that every $(i,j)$
entry of $A^ { m\tlmu }$ involves a leading \scycle, for $m$
sufficiently large, but this is clear from Remark~\ref{extract},
since $A$ is irreducible. \end{proof}

\begin{thm}\label{ghostpot1} A matrix $A$ is ghostpotent iff the
submatrix of each of its strongly connected components is
ghostpotent according to the criterion of Theorem~\ref{singp}, in
which case the index of ghostpotence is at most the number of
strongly connected components times the maximal index of
ghostpotence of the strongly connected components.
\end{thm}
\begin{proof} We write $A$ in full block triangular form, and then apply
Lemma~\ref{ghostpot}.
\end{proof}

\section{The Jordan decomposition}

We are ready to find a particularly nice form for powers of $A$.

\begin{defn}\label{blocktri} A matrix $A$ is in  \textbf{stable block triangular
form} if $$A = \(  \begin{matrix}  B_1 & B_{1,2} & \dots & B_{1,\tt -1} & B_{1,\tt } \\
 \mzero &  B_2 & \dots & B_{2,\tt -1} & B_{2,\tt }\\ \vdots & \vdots & \ddots  & \vdots
  & \vdots \\ \mzero  &   \dots &  \mzero & B_{\tt -1}
  &  B_{\tt -1,\tt }\\  \mzero&  \dots  &  \mzero  &  \mzero & B_\tt
\end{matrix}\)$$ is in full block triangular form, such that each $B_i$
is semi-idempotent and $$A^2 = \(  \begin{matrix}  \bt  _1 B_1 & \bt  _{1,2}B_{1,2} & \dots & \bt  _{1,\tt -1}B_{1,\tt -1} & \bt  _{1,\tt }B_{1,\tt } \\
 \mzero &  \bt  _{2}B_2 & \dots & \bt  _{2,\tt -1}B_{2,\tt -1} & \bt  _{2,\tt}B_{2,\tt }\\ \vdots & \vdots & \ddots  & \vdots
  & \vdots \\ \mzero  &   \dots &  \mzero & \bt  _{\tt -1}B_{\tt -1}
  &  \bt  _{\tt -1,\tt }B_{\tt -1,\tt }\\  \mzero&  \dots  &  \mzero  &  \mzero & \bt  _{\tt }B_\tt
\end{matrix}\),$$ where $\bt _i = \bt (B_i)$ and $\bt _{i,j}
\in \{ \bt _i, \bt _i^\nu  ,\dots,  \bt _j,  \bt _j^\nu  \}$ for
each $i<j.$ If each $\bt _{i,j}\in \tT,$ we say that $A$ is in
\textbf{tangibly stable block triangular form}.
\end{defn}

\begin{defn}\label{semisimple} A matrix $S$ is \textbf{semisimple}
if $S^{2k} = DS^k$ for some tangible diagonal matrix $D$ and $k
\in \Net$. We say that $A$ has a \textbf{Jordan decomposition} if
$A = S + N$ where $S$ is semisimple and $N$ is
ghostpotent.\end{defn}

Obviously, any semi-idempotent matrix is semisimple.

\begin{lem}\label{trired} Suppose   $$ A = \(  \begin{matrix}  B_1 & B_{1,2} & \dots & B_{1,\tt -1} & B_{1,\tt } \\
 \mzero &  B_2 & \dots & B_{2,\tt -1} & B_{2,\tt }\\ \vdots & \vdots & \ddots  & \vdots
  & \vdots \\ \mzero  &   \dots &  \mzero & B_{\tt -1}
  &  B_{\tt -1,\tt }\\  \mzero&  \dots  &  \mzero  &  \mzero & B_\tt
\end{matrix}\)$$ is in full block triangular form. If $B_i = S_i + N_i$
is a Jordan decomposition for $B_i$ for $i = 1, \dots, \tt $, then
$$ A = \(  \begin{matrix}  S_1 & \mzero & \dots & \mzero & \mzero \\
 \mzero &  S_2 & \dots & \mzero & \mzero \\ \vdots & \vdots & \ddots  & \vdots
  & \vdots \\ \mzero  &   \dots &  \mzero & S_{\tt -1}
  & \mzero\\  \mzero&  \dots  &  \mzero  &  \mzero & S_\tt
\end{matrix}\)+\(  \begin{matrix}  N_1 & B_{1,2} & \dots & B_{1,\tt -1} & B_{1,\tt } \\
 \mzero &  N_2 & \dots & B_{2,\tt -1} & B_{2,\tt }\\ \vdots & \vdots & \ddots  & \vdots
  & \vdots \\ \mzero  &   \dots &  \mzero & N_{\tt -1}
  &  B_{\tt -1,\tt }\\  \mzero&  \dots  &  \mzero  &  \mzero & N_\tt
\end{matrix}\)$$  is a Jordan decomposition for $A$.
\end{lem}
\begin{proof} $ $

Clearly $ \( \begin{matrix}  S_1 & \mzero & \dots & \mzero & \mzero \\
 \mzero &  S_2 & \dots & \mzero & \mzero \\ \vdots & \vdots & \ddots  & \vdots
  & \vdots \\ \mzero  &   \dots &  \mzero & S_{\tt -1}
  & \mzero\\  \mzero&  \dots  &  \mzero  &  \mzero & S_\tt
\end{matrix}\)$ is semisimple, and $\(  \begin{matrix}  N_1 & B_{1,2} & \dots & B_{1,\tt -1} & B_{1,\tt } \\
 \mzero &  N_2 & \dots & B_{2,\tt -1} & B_{2,\tt }\\ \vdots & \vdots & \ddots  & \vdots
  & \vdots \\ \mzero  &   \dots &  \mzero & N_{\tt -1}
  &  B_{\tt -1,\tt }\\  \mzero&  \dots  &  \mzero  &  \mzero & N_\tt
\end{matrix}\)$  is ghostpotent, by Lemma~\ref{ghostpot}.
\end{proof}

\begin{rem} For any matrix $A \in M_n(R)$ in full block triangular
form, we view $A$ as acting on $R^{(n)}$ with respect to the
standard basis $e_1, \dots, e_n$. The diagonal block $B_j$ uses
the columns and rows say from $i_{j}$ to $i_{j+1}-1$, and acts
naturally on the subspace $V_j$ generated by $e_{i_j}, \dots,
e_{i_{j+1}-1}.$ Thus, we have the natural decomposition $R^{(n)} =
V_1 \oplus \cdots \oplus V_\tt$.  We view each $V_j$ as a subspace
of $R^{(n)}$ under the usual embedding, and the $i_{j}$ to
$i_{j+1}-1$ columns of $A$ act naturally on $V_j$.\end{rem}

\begin{thm}\label{stab} For any matrix $A$ in full block triangular
form of length $\tt$ for which the diagonal blocks are
semi-idempotent, the matrix $A^{\tt}$ has stable block triangular
form. Furthermore, $\bt_{i,j}  \in \{ \al , \al^\nu\},$ where
$$\bt  = \sum (\bt  _{i_1} +  \bt _{i_2} + \cdots + \bt
_{i_k}),$$ summed over all paths $(i_1, \dots, i_k)$ such that
$i_1 = i$ and $i_k = j$. (Thus $k \le j -i$). In other words, $\bt
_{i,j}$ is the maximum semi-idempotent coefficient that appears in
a path from $i$ to $j$.

More generally, under the given decomposition $R^{(n)} = V_1
\oplus \cdots \oplus V_\tt$ such that, for any vector $v,$ writing
$A^\tt v = (v_1, \dots, v_\tt)$ for $v_i \in V_i,$ one has
\begin{equation}\label{eig} A(A^\tt v_j) = \sum _{i=1}^j \bbt _{i,j}
v_i,\end{equation} where $ \bbt _{i,j} $ is the maximum of the
$\bt  _{i_k}$ (or its ghost).
\end{thm}
\begin{proof} Write  $$ A = \(  \begin{matrix}  B_1 & B_{1,2} & \dots & B_{1,\tt -1} & B_{1,\tt } \\
 \mzero &  B_2 & \dots & B_{2,\tt -1} & B_{2,\tt }\\ \vdots & \vdots & \ddots  & \vdots
  & \vdots \\ \mzero  &   \dots &  \mzero & B_{\tt -1}
  &  B_{\tt -1,\tt }\\  \mzero&  \dots  &  \mzero  &  \mzero & B_\tt
\end{matrix}\) \quad \text{and} \quad
A ^{\tt } = \(  \begin{matrix}  \tlB_1  & \tlB_{1,2} & \dots & \tlB_{1,\tt -1} & \tlB_{1,\tt } \\
 \mzero &  \tlB_2 & \dots & \tlB_{2,\tt -1} & \tlB_{2,\tt }\\ \vdots & \vdots & \ddots  & \vdots
  & \vdots \\ \mzero  &   \dots &  \mzero & \tlB_{\tt -1}
  &  \tlB_{\tt -1,\tt }\\  \mzero&  \dots  &  \mzero  &  \mzero & \tlB_\tt
\end{matrix}\).$$ Then $\tlB_j = B_j ^{\tt}$ for each $j = 1,\dots, \tt$, and for
$i<j,$ \begin{equation}\label{eta} \tlB_{i,j} = \sum B_{i_1}^{\tu
_1}B_{i_1, i_2}B_{i_2}^{\tu _2}\cdots B_{i_{\ell-1},
i_\ell}B_{i_\ell}^{\tu _\ell},\end{equation} summed over all $i_1,
\dots, i_\ell$ where $i_1 = i$ and $i_\ell
 = j.$ (Here $\ell \le j-i.$) We take a typical summand
 $B_{i_1}^{\tu _1}B_{i_1, i_2}B_{i_2}^{\tu _2}\cdots
B_{i_{\ell-1}, i_\ell}B_{i_\ell}^{\tu _\ell}.$ By assumption, when
$\tu _k
>1$ we may
 replace $B_{i_k}^{\tu _k}$ by $\bt_{i_k}^{\tu _k-1}B_{i_k}.$ But then we could replace
 $B_{i_k}^{\tu _k}$ by  $B_{i_{ k'}}^{\tu _k}$ for any  $k'$.  So, taking
 $\bt = \sum _{k=1}^\ell \bt _{i_k}$, which is $\nu$-equivalent to some $\bt _{i_{k'}},$ we also have the term $$\bt ^{\tt  - \ell}
 B_{i_1, i_2}\cdots
B_{i_{\ell-1}, i_\ell}$$
  and thus we may
 assume that $\tu _k =0$ (in \eqref{eta}) for each $k \ne  k'.$ In other words, we sum over
 all paths  $p: = B_{i_1, i_2}, \dots,
B_{i_{\ell-1}, i_\ell}$, where $i_1 = i, $ $i_\ell = j,$ and the
coefficient $\bt ^{\tt  - \ell}$ comes from $B_{i_{k}} ^{\tt -
\ell},$ where $i_ k$ appears in $\tV(p).$

Note that $\ell < \tt ,$ so if there are two possibilities for
$k'$ we get two equal maximal paths and thus get a ghost value for
$\bt_{i,j}$ (and likewise if there is one maximal ghost path, or
if different maximal paths are $\nu$-equivalent). If there is one
single path $p$ of maximal weight, which is tangible, and if $p$
is tangible, then $\bt_{i,j}$ is tangible.

The same argument yields the last assertion when we consider
$A^\tt v_j$.
\end{proof}

\begin{cor}\label{stab2} Hypotheses as in Theorem~\ref{stab} $A^{2\tt}$ is in
tangibly stable block triangular form.\end{cor}
\begin{proof} All the ghost entries already occur in $A^{2\tt}$,
so we can replace any ghost $\beta_{i,j}$ by $\Inu {\beta_{i,j}}$.
\end{proof}

We now are ready for one of the major results.

\begin{thm}\label{stabil} For any matrix $A$, there is some power
$m$ such that $A^m$ is in tangibly stable block triangular
form.\end{thm}
\begin{proof} $A$ can be put into full block triangular form  by Proposition~\ref{fullbl}, and a
further power is in tangibly stable block triangular form, by
Theorem~\ref{stabil} and Corollary~\ref{semiid}.
\end{proof}

We call this $m$ the \textbf{stability index} of $A$.

\begin{example}\label{badex2} There is no bound (with respect to the matrix's size) on the stability index of  $A$. Indeed, in
Example~\ref{higho} we saw   a $2\times 2$ matrix $A$ such that
$A^{m-1}$ is nonsingular but $A^{m}$ is singular, where $m$ can be
arbitrarily large.
\end{example}

\begin{thm}\label{Jord} Any matrix $A \in M_n(R)$ has a Jordan decomposition,
where furthermore, in the notation of Definition~\ref{semisimple},
$|A| = |S|$.\end{thm}
\begin{proof} In view of Lemma~\ref{trired}, it suffices to assume
that $A$ is irreducible. Then,  we conclude with
Corollary~\ref{corepower2}.\end{proof}

\begin{example} The matrix (in logarithmic notation) $$ A = \(\begin{array}{cccc}
          10 & 10 & 9 &  - \\
          9 & 1 &  - &  - \\
          -  &  - &  -  & 9 \\
          9 &  - &  - & -  \\
        \end{array} \) $$  of \cite[Example~5.7]{IzhakianRowen2009Equations} (the empty places stand for $-\infty$) is semisimple, but the tangible matrix $B$ given there must be taken to be nonsingular.
 \end{example}

\section{Supertropical generalized eigenvectors and their
eigenvalues}\label{sec:eigenvectors}

 We started  studying supertropical eigenspaces in \cite{IzhakianRowen2008Matrices},
 and saw how to calculate supertropical eigenvectors in
\cite{IzhakianRowen2009Equations}, but also saw that the theory is
limited even when the characteristic polynomial factors into
tangible linear factors. To continue, we need to consider
generalized supertropical  eigenvectors. We recall
\cite[Definition~7.3]{IzhakianRowen2008Matrices}.

\begin{defn}\label{eigen3} A tangible  vector $v$ is a \textbf{generalized supertropical
eigenvector} of $A$, with \textbf{generalized supertropical
eigenvalue} $\bt \in \tTz$,
 if $$A^m v \lmodg \bt^m v$$ for some $m$;
the minimal such $m$ is called the \textbf{multiplicity}. A
\textbf{supertropical eigenvalue} (resp.~\textbf{supertropical
eigenvector}) is a
 generalized supertropical eigenvalue
(resp.~generalized supertropical eigenvector)  of multiplicity 1.
 A  vector
$v$ is a \textbf{strict eigenvector} of $A$, with
\textbf{eigenvalue} $\bt \in \tTz$,
 if $A v = \bt v$.


\end{defn}

Recall, cf. \cite[Definition
3.1]{IzhakianKnebuschRowen2009Linear}, that a vector $v \in
R^{(n)}$ is a \textbf{g-annihilator} of $A$ if  $Av \in
\tGz^{(n)}$, i.e., $Av \lmod \vzero$. A \textbf{tangible
g-annihilator}  is a g-annihilator that  belongs to $\tTz^{(n)}$.
(Accordingly, any tangible g-annihilator of $A$ is the same as a
supertropical eigenvector with supertropical eigenvalue $\rzero$.)
The \textbf{ghost kernel} of $A$ is defined as
$$ \Gker(A) := \{ v\in R^{(n)} \ | \ Av \in \tGz^{(n)} \}; $$
in particular $\tGz^{(n)} \subset \Gker(A)$ for any $A$. If $A$ is
a ghost matrix, then $\Gker(A) = R^{(n)}$.
\begin{example}\label{strict} Any quasi-identity matrix $A = \um_\tG$
has $n$ tropically independent strict eigenvectors, each with
eigenvalue $\rone$, namely the columns of $A$ (since $A$ is
idempotent and nonsingular). Likewise, any nonsingular
semi-idempotent matrix has $n$  tropically independent strict
eigenvectors, each with eigenvalue~$\bt(A).$
\end{example}

When $A$ is not necessarily nonsingular, we still have an
analogous result.

\begin{prop}\label{thick1} For any irreducible, semi-idempotent $n\times n$ matrix $A$,
if $s = \#( \tV(\tancore(A))),$ the $s$~columns of the submatrix
$(A)_{\tancore}$ (corresponding to $\tancore(A)$) are tropically
independent, strict eigenvectors of $(A)_{\tancore}$, and are also
supertropical eigenvectors of $A$, which can be expanded to a set
of $n$ tropically independent vectors of $R^{(n)}$, containing
$n-s$ tangible  g-annihilators of $A$.\end{prop}

\begin{proof} Replacing $A$ by ${\bt^{-1}}A$, where $\bt = \bt(A)$,  we may assume
that $A$ is an idempotent matrix.
Let $U$ denote the subspace of $R^{(n)}$ corresponding to
$(A)_{\tancore}$. If $v$ is a column of $A$, and $v' = v|_U$ is
its restriction to a column of $U$, then clearly
$$(A)_{\tancore} v' \le (A v)|_U = v' = Iv' \le (A)_{\tancore} v',$$ implying
$(A)_{\tancore} v' = (A v)|_U = v'$.

These vectors $v$ are also supertropical eigenvectors of $A$,
since the other components of $Av$ are ghost, in view of
Theorem~\ref{singp}.

To prove the last assertion, we repeat the trick of
\cite[Proposition~4.12]{IzhakianKnebuschRowen2009Linear}).
Rearranging the base, we may assume that $\tV(\tancore(A)) = \{ 1,
\dots, s \}.$ For any other row $v_u$ of $A$ ($m < u \le n$), we
have $\bt_{u,1}, \dots, \bt_{u,m}\in \tTz$ such that $v_u + \sum
\bt_{i,j} v_i \in \tGz^{(n)}.$

Let  $B'$ be the $(n-m) \, \times n$ matrix whose first $s$
columns are the $s$ columns of $(A)_{\tancore}$ (with
$(i,j)$-entry~$\rzero$ for $i>s$) and whose entries $(i,j)$ are
$\bt_{i,j}$ for $1 \leq i,j \leq m$,   and for which $\bt_{i,j} =
\delta_{i,j}$ (the Kronecker delta) for $m < j \le n$.  Then $B'$
is block triangular with two diagonal blocks, one of which is the
identity matrix, implying $|B'| = |(A)_{\tancore}|$ and thus $B'$
is nonsingular. This gives us the desired $n$ tropically
independent supertropical eigenvectors.
\end{proof}

\begin{lem}\label{getweak}
If $A^m v \gsim \bt^m v$ for a tangible vector $v$, some $m$, and
$\bt \in \tTz$, then  $v$ is a generalized supertropical
eigenvector of $A$ of multiplicity $m$, with generalized
supertropical eigenvalue $\bt$.
\end{lem} \begin{proof} The vector $\bt ^m v$ is tangible, so clearly $A^m v \lmodg \bt^m
v$ (cf.~\cite[Lemma~2.9]{IzhakianKnebuschRowen2009Linear}).
\end{proof}

\begin{lem} If $\bt$ is a generalized supertropical
eigenvalue for $A$ of multiplicity $m$, then $\bt$ also is a
generalized supertropical   eigenvalue for $A$ of multiplicity $m'$,
for each multiple $m'$ of $m$.
\end{lem}

\begin{proof}
  $$ A ^{km} v =  A ^{(k-1)m} A^m v  \lmodg A ^{(k-1)m} \bt^m v = \bt^m  A ^{(k-1)m} v  \lmodg   \bt^{km} v,$$ by
  induction.
\end{proof}

\begin{prop}\label{wksp} The  generalized supertropical eigenvectors
corresponding to a supertropical eigenvalue~$\bt$ form a subspace
$V_\bt(A) \subset R^{(n)}$ which is $A$-invariant.\end{prop}
\begin{proof} If  $v,w\in V_\bt(A),$ then
$$A^m v \lmodg \bt^m v, \qquad A^{m'} w \lmodg \bt^{m'} w,$$
for suitable $m,m'$, so taking their maximum $m''$ yields
 $A^{m''} (v+w) \lmodg \bt^{m ''}(v+w),$
and likewise for scalar products, implying $\al v\in V_\bt(A),$
for any $\al \in R$.

Also,
$$A^m(Av) = A(A^m v) \lmodg A(\bt^m v) = \bt^m (Av),$$
and thus $Av \in  V_\bt(A).$
\end{proof}

We call this space $ V_\bt(A)$ the \textbf{generalized
supertropical  eigenspace} of $\bt$. This is easiest to describe
when $A$ is nonsingular.

\begin{thm}\label{eigendec} Suppose a nonsingular matrix $A$ is in stable block triangular
form, notation as in Definition~\ref{blocktri}, and write $V = V_1
\oplus \cdots \oplus V_\tt$ where each $V_i$ has rank $n_i$ and
$$Av_j =  \sum _i B_{i,j} v_i, \quad \forall v_j \in V_j.$$
Then there are supertropical eigenspaces $\tlV_j$ of $A$ with
respect to supertropical eigenvalues $\bt _j$, such that $ V _A :=
\tlV_1 \oplus \cdots \oplus \tlV_\tt$ is a \textbf{thick} subspace
of $V$ in the sense of \cite[Definition
5.28]{IzhakianKnebuschRowen2009Linear} (which means that $ V _A$
also has rank $n$).
\end{thm}
\begin{proof} Each diagonal block $B_j$ is nonsingular. Let $V_j'$ denote the subspace of
$V_j$ spanned by the rows of~$B_j.$ In other words, $$V_j' : = \{
B_j v : v \in V \},$$   a thick subspace of $V$ in view of
\cite[Remark~6.14]{IzhakianKnebuschRowen2009Linear},
 since $B_j$   behaves like a
 quasi-identity matrix in view of Lemma~\ref{quasired}.

  Now for each $v \in
 V_j'$ we write $Av = \sum _{i=1}^\tt v_i$ where $v_i \in V_i.$ By
Theorem~\ref{stab}, $$Av_j = \sum _{i=1}^j  \bbt _{i,j} v_i$$ for
$v_j \in \hV_j$. Starting with $i = j$ we put $\tlv  _{j,j} = v_j$
and, proceeding by reverse induction, given $\tlv  _{k,j}$ for
$i<k \le j$ take
$$\tlv  _{i,j} = \sum _{k=i+1}^j \frac{\bbt _{k,j}}{\bt_i} \tlv_{k,j}.$$
We put $$\tlv _j =   \tlv  _{1,j}+ \cdots + \tlv _{j,j}.$$  Then
for each $i<j$ the $i$-component of $A \tlv _j$ is
$$\sum _{k=i+1}^j \left(  \bbt _{k,j}   \tlv_{k,j} + {\bt_i}\frac{\bbt _{k,j}}{\bt_i} \tlv _{k,j} \right)=
\sum _{k=i+1}^j  \bbt _{k,j} ^\nu    \tlv _{k,j}
 ,$$ whereas the $j$-component of $A \tlv _j$ is $\bt _j
\tlv _j$. Hence, $A \tlv  _j \lmodg \bt _j \tlv  _j$, as desired.
\end{proof}

When $A$ need not be nonsingular, we need to modify the assertion
slightly.

\begin{thm}\label{eigendec1} Suppose the matrix $A$ is in stable block triangular
form, notation as in Definition~\ref{blocktri}, and write $V = V_1
\oplus \cdots \oplus V_\tt$ where each $V_i$ has rank $n_i$ and
$$Av_j =  \sum _i B_{i,j} v_i, \quad \forall v_j \in V_j.$$ Let
$s_j = \# (\tV(\tancore(B_j))).$ Then there are supertropical
eigenspaces $\tlV_j$ of $A$ with respect to supertropical
eigenvalues $\bt _j$, as well as a g-annihilator space $V_0$, such
that $ V _A := \tlV_0 \oplus \tlV_1 \oplus \cdots \oplus \tlV_\tt$
is a  thick  subspace of $V$.
\end{thm}
\begin{proof} We repeat the proof of Theorem~\ref{eigendec} noting that when $B_j$ is singular,
 one could take $\hB_j$ to be the space of
 Proposition~\ref{thick1}, which provides extra g-annihilating  vectors in each component, but does not affect the rest of the
argument.
\end{proof}

\subsection{Weak generalized supertropical eigenspaces }
 Generalized eigenspaces are understood better when we
introduce the following related notion. (We write $\vzero$ for the
zero vector in $R^{(n)}$.)

\begin{defn}\label{eigen4} A  vector $v\ne \vzero$ is a \textbf{weak generalized supertropical  eigenvector} of $A$,
with (tangible) \textbf{weak $m$-generalized supertropical
eigenvalue} $\bt \in \tTz$,
 if $$(A^m + \bt^m I) ^k v \lmodg \vzero$$   for some~$k$.
\end{defn}

\begin{rem} Any generalized supertropical  eigenvector is a weak $m$-generalized supertropical
eigenvector, in the view of Remark~\ref{Frob}.
\end{rem}
\begin{lem} If $\bt $ is a weak $m$-generalized supertropical
eigenvalue for $A$, then $\bt $ also is a weak $m'$-generalized
supertropical   eigenvalue for $A$, for each $m'$ dividing $m$.
\end{lem}

\begin{proof} Write $m = m'd$.
  $$ (A ^{m'} + \bt ^{m'} I )^d   \lmodg A^{m'd} + \bt^{m'd} I = A^{m } + \bt^{m } I  ,$$ by Proposition~\ref{Frob},
  yielding $ (A ^{m'} + \bt ^{m'} I )^{dk} v  \lmodg (A^{m } + \bt^m I)^k v $.
Thus, $v$ is a weak $m'$-generalized supertropical  eigenvector
for $A$.
\end{proof}

Just as with Proposition~\ref{wksp}, we have (with the analogous
proof):
\begin{prop}\label{wksp1} The weak $m$-generalized supertropical
eigenvectors 
corresponding to a root $\bt$ form an  $A$-invariant subspace of $
R^{(n)}$.\end{prop}

We call this space  of Proposition~\ref{wksp1} the \textbf{weak
$m$-generalized eigenspace} of $\bt$.
 Considering $A$ as a linear operator acting
on $R^{(n)}$, the weak $m$-generalized eigenspace is the union of
the ascending chain of subspaces
$$ \Gker (A^m + \bt^m  I)  \subseteq   \Gker (A ^m + \bt ^m I)^2 \subseteq \cdots .$$

The following technique gives us a method to compute weak
generalized eigenvectors.

\begin{rem}\label{eigens} Suppose $A^m$ satisfies a polynomial $f = \prod f_i,$
where each $f_i$ is monic $a_i$-primary with constant term
$\beta_i^{n_i}$, and
 for each $1 \le j \le t$ let $$ g_j = \prod _{i \ne j}
f_i  = \frac {f}{f_j}.$$ Then for each $v \in
  g_j(A)R^{(n)},$ $$(A^m+ \bt_jI)^{n_j}v  \lmodg f(A^m)v \lmodg
  \fzero,$$
implying $v$ is a weak $m$-generalized eigenvector of $A$, with
eigenvalue $\bt_j$.  \end{rem}

This gives us a weak $m$-generalized eigenspace of $A$ clearly
containing the generalized eigenspace $ V_{\bt_j}(A)$, and leads
us to explore the connection between these two notions.

\begin{lem}\label{weak1} Suppose  $v $ is a weak
m-generalized supertropical eigenvector of an irreducible matrix
$A$ of stability index $m'$, and supertropical eigenvalue $\bt $.
Suppose $q = dm' ,$ and suppose $A^{2m'} = \gamma A.$ Let
$$v' = \sum _{j  =0}^{m'-1} A^{j}\bt ^{(m'-j)}v.$$
 Then  $$\sum _{j =0}^{q} (A +\bt  I)^j v = \begin{cases}
\bt  ^ d v' &  \text{for}\quad \bt  >_{\nu} \gamma, \\
(\bt ^d)^\nu v' & \text{for}\quad \bt   \nucong \gamma, \\ \gamma
^d v' & \text{for}\quad \bt  <_{\nu} \gamma. \end{cases}$$
\end{lem}
\begin{proof}  This is  immediate from Proposition~\ref{Frob1}. \end{proof}

We are ready to show that the behavior of weak generalized
supertropical eigenvectors is ``controlled'' by the stability
index of $A$.

\begin{thm}\label{weak2} Given a matrix $A$ with
 stability  index $m'$,  suppose $v \in V_A$  is a
weak generalized supertropical eigenvector of $A$ with weak
generalized supertropical eigenvalue $\bt.$ Then $\bt$ is a
generalized supertropical eigenvalue of $A^{m'}$, and $\sum _{j
=0}^{2m'} (A +\bt I)^j v $ is already a ghost vector.\end{thm}
\begin{proof}  Decomposing $v$ as in Theorem~\ref{stab}, we may assume that $v$
is some $V_j$. Take $v'$ as in Lemma~\ref{weak1}. If $\bt \ne
\bt_j,$  then we get a tangible component in $(A+ \bt I)^q v$ for
high enough powers of $q$, contrary to assumption. Hence $\bt =
\bt_j,$ and again we conclude with Lemma~\ref{weak1}.
\end{proof}

So we see that the ``difference'' between weak generalized
supertropical eigenvalues and generalized supertropical
eigenvalues occurs within twice the stability index. (We could
lower this bound with some care.)

\begin{example}\label{badex1} $ $

 Here is an example which illustrates some new pitfalls. We take
$A =   \( \begin{matrix} 0 & 0 \\
1 & 2
\end{matrix}\)$ as in
Example~\ref{badex}. Clearly $0 = 0^2$ and $4 = 2^2$ are
supertropical eigenvalues of $A^2,$ but now, in view of
\cite[Proposition 7.7]{IzhakianRowen2008Matrices}, every tangible
$\bt  \le _\nu 1$ is a supertropical eigenvalue of $$A^2 =   \( \begin{matrix} 1 & 2 \\
3 & 4
\end{matrix}\),$$ since $\bt $ is a
root of $f_{A^2} = \la^2 + 4\la + 5^\nu.$ Let us compute
 the tangible eigenvectors, using the methods of~\cite{IzhakianKnebuschRowen2009Linear}.

The singular matrix $A^2$ has adjoint $  \( \begin{matrix} 4 & 2 \\
3 & 1
\end{matrix}\)$ and thus the g-annihilator
$v= (2,1)^{\trn}$, which can be checked by noting that $A^2 v =
(3^\nu, 5^\nu)^{\trn},$ which is ghost. From this point of view,
$(2,1)^{\trn}$ is a generalized supertropical eigenvector  for $A$
having eigenvalue $-\infty$ of multiplicity 2, although it is also
a g-annihilator of $A^2$.

Note that $A^2 + \bt I = A^2$ for all $\bt <_\nu 1.$ From this
point of view, these $\bt$ are ``phony'' generalized eigenvalues
of $A$.\end{example}



\end{document}

%% file: tropMacro.tex
\usepackage{epsfig,latexsym,amsfonts,amssymb,amsmath,amscd,graphics,epic}
\usepackage{amsfonts,amssymb,amsmath,amscd,amsthm}
\usepackage[mathscr]{eucal}
\usepackage{mathrsfs}
\usepackage{oldgerm,units}
\usepackage{wrapfig,epsfig}

\usepackage{ifthen}
\usepackage{mathbbol}

\usepackage{amsthm}
\usepackage[mathscr]{eucal}
\usepackage{mathrsfs}
\usepackage{mathbbol}
\usepackage{oldgerm,units}
\usepackage{wrapfig}

\newtheorem{theorem}{Theorem}[section]

\newtheorem{example}[theorem]{Example}

\newcommand{\Real}{\mathbb R}

\newcommand{\Net}{\mathbb N}

\newcommand{\one}{\mathbb{1}}
\newcommand{\zero}{\mathbb{0}}

\newcommand{\trop}[1]{\mathcal{#1}}

\newcommand{\tG}{\trop{G}}

\newcommand{\tT}{\trop{T}}

\newcommand{\tV}{\trop{V}}

\newcommand{\om}{\omega}

\newcommand{\al}{\alpha}
\newcommand{\bt}{\beta}

\newcommand{\lm}{\lambda}

\newcommand{\id}{\inva{\aad}}

    \ifx\proof\undefined
    \newenvironment{proof}{
    \smallskip
    \noindent\emph{Proof.}}{\hfill\(\Box\)
    \bigskip
    } \fi

\newcommand{\ifdef}[3]{\ifthenelse{\equal{#1}{true}}{#2}{#3}}